\documentclass[10pt, oneside, a4paper]{article}
\usepackage{amsmath, amsthm, amssymb, amsfonts}
\usepackage{geometry}
\usepackage{graphicx}
\usepackage{natbib}
\usepackage{setspace}
\usepackage{xcolor}
\usepackage[colorlinks, pdfstartview=FitH, citecolor = black, linkcolor  = black]{hyperref}
\usepackage{mathtools}  
\mathtoolsset{showonlyrefs}  
\usepackage{commath}
\usepackage{multicol}
\numberwithin{equation}{section}
\theoremstyle{plain}
\newtheorem{assumption}{Assumption}
\newtheorem{theorem}{Theorem}[section]

\newtheorem{lemma}[theorem]{Lemma}
\newtheorem{proposition}[theorem]{Proposition}

\theoremstyle{definition}
\newtheorem{definition}{Definition}[section]
\newtheorem{example}{Example}
\newtheorem{remark}[theorem]{Remark}

\newcommand{\N}{\mathbb{N}}
\newcommand{\R}{\mathbb{R}}

\newtheorem*{continuancex}{Example \continuanceref \text{ }continued}
\newenvironment{continuance}[1]
  {\newcommand\continuanceref{\ref{#1}}\continuancex}
  {\endcontinuancex}

\begin{document}

\title{
Power in High-Dimensional Testing Problems
\footnote{
Financial support by the Danish National Research Foundation (Grant DNRF 78, CREATES) is gratefully acknowledged; David Preinerstorfer's research was partly supported by the Program of Concerted Research Actions (ARC) of the Universit\'e libre de Bruxelles.
We are grateful to the Co-Editor
and three anonymous referees for comments and suggestions which greatly helped to improve this paper.
We would also like to thank Marc Hallin, Michael Jansson, Davy Paindaveine, Lukas Steinberger and Michael Wolf for helpful comments and discussions.	
}
}

\author{
\begin{tabular}{c}
Anders Bredahl Kock \\ 	
{University of Oxford} \\
{CREATES, Aarhus University} \\
{\href{mailto:anders.kock@economics.ox.ac.uk}{anders.kock@economics.ox.ac.uk}} 
\end{tabular}
\and
\begin{tabular}{c}
David Preinerstorfer \\ 
{ECARES} \\ 
{Universit\'e libre de Bruxelles} \\ {\href{mailto:david.preinerstorfer@ulb.ac.be}{david.preinerstorfer@ulb.ac.be}}
\end{tabular}
}

\bigskip

\date{
\begin{tabular}{c}
First version: September 2017 \\
This version: January 2019
\end{tabular}
}

\maketitle

\begin{abstract}
	\cite{fan2015} recently introduced a remarkable method for increasing asymptotic power of 
	tests in high-dimensional testing problems. If 
	applicable to a given test, 
	their \textit{power enhancement principle} leads to 
	an improved test that has the same asymptotic size, uniformly non-inferior asymptotic power, and is consistent against a strictly broader range of alternatives than the initially given test. 
We study under which conditions this method can be applied and show the following: In asymptotic regimes where the dimensionality of the parameter space is fixed as sample size increases, there often exist tests that can not be further improved with the power enhancement principle. 
However, when the dimensionality of the parameter space increases sufficiently slowly with sample size and a marginal local asymptotic normality (LAN) condition is satisfied, \textit{every} test with asymptotic size smaller than one can be improved with the power enhancement principle. While the marginal LAN condition alone does not allow one to extend the latter statement to all rates at which the dimensionality increases with sample size, we give sufficient conditions under which this is the case.
\end{abstract}

\vspace{0.5cm}

\section{Introduction}\label{sec:intro}
The effect of dimensionality on power properties of tests has witnessed a lot of research in recent years. One common goal is to construct tests with good asymptotic size and power properties for testing problems where the length of the parameter vector involved in the hypothesis to be tested increases with sample size.
In the context of high-dimensional cross-sectional testing problems \cite{fan2015} introduced a \textit{power enhancement principle}, which essentially works as follows: given an initial test, one tries to find another test that has asymptotic size zero and is consistent against sequences of alternatives the initial test is not consistent against. If such an auxiliary test, a \textit{power enhancement component} of the initial test, can be found, one can construct an improved test that has better asymptotic properties than the initial test. In particular, one can obtain a test that (i) has the same asymptotic size as the initial test, (ii) has uniformly non-inferior asymptotic power when compared to the initial test, and (iii) is consistent against all sequences of alternatives the auxiliary test is consistent against. As a consequence of (iii) the improved test is consistent against sequences of alternatives the initial test is not consistent against. \cite{fan2015} illustrated their power enhancement principle by showing how an initial test based on a weighted Euclidean norm of an estimator can be made consistent against sparse alternatives, which it could previously not detect, by incorporating a power enhancement component based on the supremum norm of the estimator. The existence of a suitable power enhancement component in the specific situation they consider, however, does not answer the following general questions:

\begin{itemize}
\item Under which conditions does a test admit a power enhancement component?
\item And, similarly, do there exist tests for which no power enhancement components exist?
\end{itemize}

In this paper we address these questions in a general setup. In the sequel we call tests that do (not) admit a power enhancement component \textit{asymptotically (un)enhanceable}. We first consider the classic asymptotic setting, where the dimension of the parameter vector being tested remains fixed as the sample size tends to infinity. Under fairly weak assumptions on the model we prove (cf.~Theorem \ref{thm:finited}) that in this asymptotic regime tests exist that are asymptotically unenhanceable. That is, in such settings there exist tests that can not be further improved by the power enhancement principle. Furthermore, such tests exist with any asymptotic size $\alpha \in (0, 1]$. Moreover, Wald-type tests often turn out to be asymptotically unenhanceable under weak regularity conditions (cf.~Theorem \ref{thm:finited2}). The situation changes drastically when the dimension increases (unboundedly) with the sample size. Here we show (cf.~Theorem \ref{thm:main2}) that if the models under consideration satisfy a mild ``fixed-dimensional'' (i.e., ``marginal'') local asymptotic normality (LAN) assumption, then for all sufficiently slowly increasing growth rates of the dimension of the parameter vector \textit{every} test with asymptotic size less than one is asymptotically enhanceable. We stress that these growth rates can be chosen to be arbitrarily slow, but can \textit{not} be chosen to be arbitrarily rapid in general. Two aspects of this result may be somewhat surprising: Firstly, one may have conjectured that the behavior from the fixed-dimensional case carries over to growth rates that diverge ``slowly enough'', which, however, is not the case. Secondly, given the non-existence of asymptotically unenhanceable tests for slow growth rates, it would be natural to expect that the fixed-dimensional behavior breaks down also for growth rates that diverge ``quickly''. While this is often correct, the marginal LAN condition used in Theorem \ref{thm:main2} is not sufficient to conclude such a behavior, as we demonstrate in our Example \ref{ex:counter}. Guided by this example, we then introduce and discuss a fairly natural additional assumption, under which we finally show in Theorem \ref{thm:main} that for \textit{any} growth rate of the dimension of the parameter vector \textit{every} test of asymptotic size less than one is asymptotically enhanceable. 

We would like to stress that even if a test is asymptotically enhanceable, the test might still be ``optimal'' within a restricted class of tests (e.g., satisfying certain invariance properties), or the test might still have ``optimal detection properties'' against certain subsets of the alternative (some of the  references given in Section~\ref{sec:literature} below establish ``optimality'' properties of this type). Hence, our findings are not in contradiction with such results. Instead, our results provide an alternative perspective on power properties in high-dimensional testing problems.

In the subsequent section we shall illustrate our findings in the context of a Gaussian location model. The results obtained are special cases of the general results developed in this article. To develop some intuitive understanding, however, we shall provide direct arguments, which exploit specific properties of the Gaussian location model.

\subsection{Asymptotic enhanceability in Gaussian location models}\label{sec:Gloc}

We denote the $m$-variate Gaussian distribution with mean $\mu$ and covariance matrix $\Sigma$ by $N_m(\mu, \Sigma)$, and throughout this section we fix the level of significance $\alpha \in (0, 1)$. Let $X_1, \hdots, X_n$ be i.i.d.~with $X_i \sim N_{d(n)}(\theta, I_{d(n)})$ for $i = 1, \hdots, n$, and where $I_m$ denotes the $m$-dimensional identity matrix. We want to test whether the unknown mean vector $\theta \in \R^{d(n)}$ equals zero. Define $Z_n := n^{-1/2} \sum_{i = 1}^n X_i \sim N_{d(n)}(\sqrt{n} \theta, I_{d(n)})$ and note that $Z_n$ is a sufficient statistic for $\theta$. Let $\phi_n$ be the test that rejects if $\|Z_n\|_2^2$ exceeds the $1-\alpha$ quantile of the $\chi^2$-distribution with $d(n)$ degrees of freedom. Clearly, $\phi_n$ has size $\alpha$ for all $n \in \N$. 

In a setting where $d(n)$ is fixed to some $d \in \N$ as sample size $n \to \infty$, the sequence of tests $\phi_n$ is consistent against a sequence of alternatives $\theta_n$ if and only if $\sqrt{n} \|\theta_n\|_2 \to \infty$. A contiguity argument now shows that there does \emph{not} exist a sequence of tests of asymptotic size $0$, which is consistent against a sequence of alternatives $\phi_n$ is not consistent against.\footnote{Suppose the sequence of tests $\nu_n$ has asymptotic size $0$ and $\sqrt{n} \|\theta_n\|_2 \not \to \infty$. Then, $N_{d}(0, I_d)$ and $N_{d}(\sqrt{n} \theta_n, I_d)$ are mutually contiguous along a subsequence $n'$. Hence, along $n'$ the power of $\nu_n$ against $\theta_n$ converges to $0$.} Therefore, in the regime $d(n) = d$ the sequence $\phi_n$ does not admit a power enhancement component and is thus asymptotically unenhanceable (using the terminology introduced in the previous section).

When $d(n)$ diverges increasingly to $\infty$, $\phi_n$ is consistent against a sequence~$\theta_n$ if and only if $d(n)^{-1/2} n  \|\theta_n\|_2^2 \to \infty$ (cf.~Lemma~\ref{lem:auxGaussloc} in  Section~\ref{sec:consreg}). In contrast to above, one can now construct a sequence of tests with asymptotic size $0$, and which is consistent against a sequence of alternatives that $\phi_n$ is not consistent against: Consider, for example, $\theta_n = a_n e_1(d(n))$ where  $a_n = \sqrt{\log(d(n))/(2n)}$ and $e_i(m)$ denotes the $i$-th element of the canonical basis in $\R^m$. Note that $\phi_n$ is not consistent against~$\theta_n$. Nevertheless, the test $\nu_n$, which rejects if the z-test statistic $Z_n^{(1)}=e_1(d(n))' Z_n \sim N(\sqrt{n} \theta_n^{(1)}, 1)$ is in absolute value greater than $(\sqrt{n}a_n)^{1/2}=[\log(d(n))/2]^{1/4}$, has asymptotic size $0$ and is consistent against $\theta_n$. This implies that in case $d(n)$ diverges to $\infty$, regardless at which rate, $\phi_n$ is asymptotically enhanceable, because it admits the enhancement component $\nu_n$. Clearly, this argument relied on properties of the specific test $\phi_n$ under consideration. To see that \emph{any} sequence of tests $\varphi_n$, say, with asymptotic size $\alpha$ is asymptotically enhanceable is more involved. We argue as follows: (a) by a sufficiency argument we can assume that $\varphi_n$ depends on $X_1, \hdots, X_n$ only via the sufficient statistic $Z_n \sim N_{d(n)}(\sqrt{n} \theta, I_{d(n)})$; (b) arguing as in Section 3.3.7 of \cite{ingster} 
delivers the following: let $\mathbb{Q}_{n,0} = N_{d(n)}(0, I_{d(n)})$, i.e., the distribution of $Z_n$ under the null, and define the mixture $\mathbb{Q}_n = d(n)^{-1} \sum_{i = 1}^{d(n)} \mathbb{Q}_{n,i}$, where  $\mathbb{Q}_{n,i} = N_{d(n)}(\sqrt{n} a_n e_{i}(d(n)), I_{d(n)})$ is the distribution of $Z_n$ under the alternative $a_n e_i(d(n))$. Note that the likelihood-ratio statistic of $\mathbb{Q}_n$ w.r.t.~$\mathbb{Q}_{n,0}$ is given by $L_n = d(n)^{-1}\sum_{i = 1}^{d(n)} e^{\sqrt{n}a_n z_i -n a_n^2/2}$. Denote the expectation operators w.r.t.~$\mathbb{Q}_n$ and $\mathbb{Q}_{n,i}$ by $\mathbb{E}^Q_{n}$ and $\mathbb{E}^Q_{n,i}$, respectively. Then, 
%
\begin{equation}\label{eq:gauss1}
	\big|\mathbb{E}_{n,0}^Q(\varphi_n) - d(n)^{-1} \sum_{i = 1}^{d(n)} \mathbb{E}_{n,i}^Q(\varphi_n)\big|^2 
	= \left|\mathbb{E}^Q_{n,0}(\varphi_n(1 - L_n))\right|^2 \leq \mathbb{E}^Q_{n,0}\left((1 - L_n)^2\right) = \mathbb{E}^Q_{n,0}\left(L_n^2\right) -1,
\end{equation}
where we used $d(n)^{-1} \sum_{i = 1}^{d(n)} \mathbb{E}_{n,i}^Q(\varphi_n) = \mathbb{E}_{n}^Q(\varphi_n) = \mathbb{E}_{n,0}^Q(\varphi_nL_n)$, Jensen's inequality and $\mathbb{E}_{n,0}^Q(L_n) = 1$. From the moment-generating-function of a normal distribution we obtain
\begin{equation}\label{eq:gauss2}
	\mathbb{E}^Q_{n,0}\left(L_n^2\right) -1 = d(n)^{-2} \sum_{i = 1}^{d(n)} \sum_{j = 1}^{d(n)} \mathbb{E}^Q_{n,0}\left(e^{\sqrt{n}a_n(z_i+z_j) -n a_n^2}\right) - 1 \leq d(n)^{-1/2} \to 0.
\end{equation}
Therefore, existence of a sequence $\theta_n = a_n e_{i(n)}(d(n))$ against which $\varphi_n$ has asymptotic power at most $\alpha$ follows. (c) Finally, observe that the test $\nu_n$, say, which rejects if the z-test statistic $Z_n^{(i(n))} = e_{i(n)}(d(n))'Z_n$ exceeds $(\sqrt{n} a_n)^{1/2}$ in absolute value has asymptotic size $0$ and is consistent against this sequence $\theta_n$. Hence $\varphi_n$ permits the enhancement component $\nu_n$ and is thus asymptotically enhanceable. Note that this holds for \emph{any} rate at which $d(n)$ diverges to $\infty$.

\subsection{Related literature}\label{sec:literature}

The setting we consider in our main results (Theorems \ref{thm:main2} and \ref{thm:main}) requires neither independent nor identically distributed data, and covers many situations of practical interest. On the other hand, for concrete high-dimensional testing problems, and under suitable assumptions on how fast the dimension of the parameter to be tested is allowed to increase with sample size, many articles have considered the construction of tests with good size and power properties. For a discussion of several examples in the context of financial econometrics (testing implications from multi-factor
pricing theory) or panel data models (tests for cross-sectional independence in mixed effect panels) we refer to \cite{fan2015}. 
Testing problems in one- or two-sample multivariate location models, i.e., tests for the hypothesis whether the mean vector of a population is zero, or whether the mean vectors of two populations are identical, were analyzed in \cite{dempster}, \cite{bai1996}, \cite{srivastava2008}, 
\cite{srivastava2013}, \cite{tony2014two}, and \cite{chakraborty2017}; in this context the articles \cite{pinelis2010asymptotic, pinelisI}, where the asymptotic efficiency of tests based on different $p$-norms relative to the Euclidean-norm were studied, need to be mentioned. For some concrete examples of high-dimensional location models arising in empirical economics see the discussion in Section 2.1 of \cite{kasy}. In regression models power properties of F-tests when the dimension of the parameter vector increases with sample size have been investigated in \cite{wang2013}, \cite{zhong2011}, and \cite{steinberger2016}. A question of particular interest in regression models is whether \emph{any} of the regressor coefficients is different from zero. This results in a hypothesis that involves, potentially after recentering, all parameters of the (high-dimensional) model. Tests for such problems are among the standard output in any econometric software package, and are routinely used in the context of model specification and selection and in judging the explanatory power of a given model. In the context of testing hypotheses on large covariance matrices, e.g., its diagonality or sphericity, properties of tests were studied in \cite{ledoit2002some}, \cite{srivastava2005}, \cite{baijiang2009}, and \cite{onatski2013, onatski2014}. For properties of tests for high-dimensional testing problems arising in
spatial statistics  
we refer to 
\cite{cai2013distributions}, \cite{ley2015}, and \cite{cutting2017testing}. These testing problems all concern a ``high-dimensional'' parameter the dimension of which increases with sample size. 

An article that obtained results somewhat similar to ours is \cite{janssen2000},
where local power properties of goodness-of-fit tests (cf.~also \cite{ingster}) are studied. For such testing problems it was shown, among other things, that any test can have high local asymptotic power relative to its asymptotic size only against alternatives lying in a finite-dimensional subspace of the parameter space. Although related, our results are qualitatively different, because asymptotic enhanceability is an intrinsically non-local concept (cf.~Remark \ref{rem:enh}), and because we do not consider testing problems with infinitely many parameters for any sample size. Instead we consider situations where the number of parameters can increase with sample size at different rates. Results on power properties of tests in situations where the sample size is fixed while the number of parameters diverges to infinity have been obtained in \cite{lockhart2016}, who showed that in such scenarios the asymptotic power of invariant tests (w.r.t.~various subgroups of the orthogonal group) against contiguous alternatives coincides with their asymptotic size. 
\section{Framework}\label{sec:fram}

The general framework in this article is a double array of experiments
\begin{equation}\label{eqn:darray}
\left(\Omega_{n, d}, \mathcal{A}_{n, d}, \{\mathbb{P}_{n, d, \theta}: \theta \in \Theta_d\}\right) \quad \text{ for } n \in \N \text{ and } d \in \N,
\end{equation}
where for every $n \in \N$ and $d \in \N$ the tuple $(\Omega_{n, d}, \mathcal{A}_{n, d})$ is a measurable space, i.e., the sample space, and $\{\mathbb{P}_{n, d, \theta}: \theta \in \Theta_d\}$ is a set of probability measures on that space, i.e., the set of possible distributions of the data observed. For every $d \in \N$ the parameter space $\Theta_d$ is assumed to be a subset of $\R^{d}$ that contains a neighborhood of the origin. The two indices $n \in \N$ and $d \in \N$ should be interpreted as ``sample size'' and as the ``dimension of the parameter space'', respectively. The expectation w.r.t. $\mathbb{P}_{n, d, \theta}$ is denoted by $\mathbb{E}_{n, d, \theta}$.

We consider the situation where one wants to test (possibly after a suitable re-parameterization) whether or not the unknown parameter vector $\theta$ equals zero. Such problems have been studied extensively in the classic asymptotic framework where $d$ is \textit{fixed} and $n \to \infty$, i.e., properties of sequences of tests for the testing problem
\begin{equation}\label{eqn:testprobfix}
H_0: \theta = 0 \in \Theta_{d} \quad \text{ against } \quad H_1: \theta \in \Theta_{d}\setminus \{0\}
\end{equation}
are studied in the sequence of experiments
\begin{equation}\label{eqn:expfix}
\left(\Omega_{n, d}, \mathcal{A}_{n, d}, \{\mathbb{P}_{n, d, \theta}: \theta \in \Theta_d\}\right) \quad \text{ for } n \in \N.
\end{equation}
In contrast to such an analysis, the framework we are interested in is the more general situation in which $d = d(n)$ is a non-decreasing sequence. More precisely, we study properties of sequences of tests for the sequence of testing problems
\begin{equation}\label{eqn:testprob}
H_0: \theta = 0 \in \Theta_{d(n)} \quad \text{ against } \quad H_1: \theta \in \Theta_{d(n)}\setminus \{0\}
\end{equation}
in the corresponding sequence of experiments
\begin{equation}\label{eqn:exp}
(\Omega_{n, d(n)}, \mathcal{A}_{n, d(n)}, \{\mathbb{P}_{n, d(n), \theta}: \theta \in \Theta_{d(n)}\}) \quad \text{ for } n \in \N,
\end{equation}
for all possible rates $d(n)$ at which the dimension of the parameter space can increase with $n$. The following running example illustrates our framework (for several more examples we refer to the detailed treatment of the Gaussian location model in Section \ref{sec:Gloc} and to the end of Section \ref{subsec:ex}). Here, for $j \in \N$, $\lambda_j$ denotes Lebesgue measure on the Borel sets of $\R^j$.

\begin{example}[Linear regression model]\label{ex:linreg}
Consider the linear regression model $y_i=x_{i,d} '\theta+u_i,\ i=1,...,n$ where $\theta \in \Theta_d = \R^d$. One must distinguish between the cases where the covariates $x_{i,d}$ are fixed or random:

\noindent \textit{Fixed covariates:} Here the sample space $\Omega_{n,d}$ equals $\R^n$, $\mathcal{A}_{n,d}$ is the corresponding Borel $\sigma$-field, and $x_{i,d} = (x_{1i}, \hdots, x_{di})'$ for $\mathsf{X} = (x_{kl})_{k,l = 1}^{\infty}$ a given double array of real numbers. Assuming that the error terms $u_i$ are i.i.d.~with $u_1\sim F$ having $\lambda_1$-density $f$, it follows that $y_i$ has $\lambda_1$-density $g_{i}(y)=f(y-x_{i,d}'\theta)$, and $\mathbb{P}_{n,d,\theta}$ is the corresponding product measure. 

\noindent \textit{Random covariates:} Here the sample space $\Omega_{n,d}$ equals $\bigtimes_{i=1}^n(\R\times \R^{d})$ and $\mathcal{A}_{n,d}$ is the corresponding Borel $\sigma$-field. Letting the error terms be as in the case of fixed covariates, and the $(u_i, x_{i,d})$ be i.i.d.~and so that $x_{1,d}$ is independent of $u_1$ with distribution $K_d$ on the Borel sets of $\R^d$, we have that $\mathbb{P}_{n,d,\theta}$ is the $n$-fold product of the measure with density $f(y-x'\theta)$ w.r.t.~$(\lambda_1\otimes K_d)(y,x)$.
\end{example}

\section{Asymptotic enhanceability}\label{sec:aden}

After re-formulating the main idea in \cite{fan2015} in terms of tests instead of test statistics, a corresponding power enhancement principle can be formulated in our general context: Let $d(n)$ be a non-decreasing sequence of natural numbers and let $\varphi_n: \Omega_{n, d(n)} \to [0, 1]$ be measurable, i.e., $\varphi_n$ is a sequence of tests for \eqref{eqn:testprob} in \eqref{eqn:exp}. Suppose that it is possible to find another sequence of tests $\nu_n: \Omega_{n, d(n)} \to [0, 1]$ with \textit{asymptotic size} $0$, i.e.,
\begin{equation}\label{eqn:asyref}
\limsup_{n \to \infty} \mathbb{E}_{n, d(n), 0} (\nu_n) = 0,
\end{equation}
and such that $\nu_n$ is consistent against at least one sequence $\theta_n \in \Theta_{d(n)}$ which the initial test $\varphi_n$ is not consistent against, i.e., 
\begin{equation}\label{eqn:enhag}
1 = \lim_{n \to \infty} \mathbb{E}_{n, d(n), \theta_n}(\nu_n) > \liminf_{n \to \infty} \mathbb{E}_{n, d(n), \theta_n}(\varphi_n).
\end{equation}
In this case $\varphi_n$ and $\nu_n$ can be combined into the test 
\begin{equation}\label{eqn:PEP}
\psi_n = \min(\varphi_n + \nu_n, 1),
\end{equation}
which has the following properties (as is easy to verify):

\begin{enumerate}
\item $\psi_n$ has the same asymptotic size as $\varphi_n$.
\item $\psi_n \geq \varphi_n$, implying that $\psi_n$ has nowhere smaller power than $\varphi_n$. 
\item $\psi_n$ is consistent against the sequence of alternatives $\theta_n$ (which $\varphi_n$ is not consistent against).
\end{enumerate}

This method of obtaining a sequence of tests $\psi_n$ with improved asymptotic properties from a given sequence $\varphi_n$ is applicable whenever $\nu_n$ with the above properties can be determined. A sequence of tests $\varphi_n$ for which there exists such a corresponding sequence of tests $\nu_n$, i.e., an \textit{enhancement component}, will subsequently be called \textit{asymptotically enhanceable}. For simplicity, this is summarized in the following definition.
\begin{definition}\label{def:enh}
Given a non-decreasing sequence $d(n)$ in $\N$, a sequence of tests $\varphi_n: \Omega_{n, d(n)} \to [0, 1]$ is called asymptotically enhanceable, if there exists a sequence of tests $\nu_n: \Omega_{n, d(n)} \to [0, 1]$ and a sequence $\theta_n \in \Theta_{d(n)}$ such that \eqref{eqn:asyref} and \eqref{eqn:enhag} hold. The sequence $\nu_n$ will then be called an enhancement component of $\varphi_n$.
\end{definition}
Before we formulate our main question, we make three observations:
\begin{remark}\label{rem:enh}
Any sequence $\theta_n$ as in Definition \ref{def:enh} must be such that $\mathbb{P}_{n, d(n), \theta_{n}}$ and $\mathbb{P}_{n, d(n), 0}$ are not contiguous; in fact, must be such that for any subsequence $n'$ of $n$ the measures $\mathbb{P}_{n', d(n'), \theta_{n'}}$ and $\mathbb{P}_{n', d(n'), 0}$ are not contiguous. Hence, asymptotic enhanceability as introduced in Definition \ref{def:enh} is a ``non-local'' property in the sense that whether or not a sequence of tests can be asymptotically enhanced, depends only on its power properties against sequences of alternatives $\theta_n$ such that $\mathbb{P}_{n, d(n), \theta_{n}}$ and $\mathbb{P}_{n, d(n), 0}$ are not contiguous along any subsequence $n'$ of~$n$.
\end{remark}

{\color{black}
\begin{remark}\label{rem:testable}
For $d(n)$ a nondecreasing sequence of natural numbers, call a sequence $\theta_n$ ``asymptotically distinguishable from the null'' if there exists a sequence of tests $\nu_n$ such that $$\lim_{n \to \infty} \mathbb{E}_{n, d(n), 0} (\nu_n) = 0 \quad \text{ and } \quad \lim_{n \to \infty} \mathbb{E}_{n, d(n), \theta_n}(\nu_n) = 1$$ hold. From Definition \ref{def:enh} it is then easily seen that a test $\varphi_n$ is \emph{not} asymptotically enhanceable if and only if it is consistent against \emph{all} sequences $\theta_n$ that are asymptotically distinguishable from the null. In testing problems that are ``asymptotically non-testable'' in the sense that \emph{no} sequence of alternative exists that is asymptotically distinguishable from the null, no test can be asymptotically enhanceable. The problems we consider in our theorems are not of this degenerate type, i.e., they are ``asymptotically testable'';  cf.~also the discussion surrounding Example~\ref{ex:counter}.\footnote{More specifically, in Theorems \ref{thm:finited} and \ref{thm:finited2} this is seen to follow immediately from the assumptions imposed; and for the remaining results note that the conclusion of any sequence of tests of size smaller than $1$ being asymptotically enhanceable rules out asymptotic non-testability of the respective sequence of testing problems.}
\end{remark}	
}

\begin{remark}
In the context of Definition \ref{def:enh} one could argue that instead of \eqref{eqn:asyref} one should require the stronger property $\mathbb{P}_{n, d(n), 0}(\nu_n = 0) \to 1$ as $n \to \infty$ (hoping for ``better'' size properties of the enhanced test $\min(\varphi_n + \nu_n, 1)$ in finite samples). In particular, the constructions in \cite{fan2015} (formulated in terms of test statistics) are based on a corresponding property. But note that if $\nu_n$ is a sequence of tests as in Definition \ref{def:enh}, the sequence $\nu_n^* = \mathbf{1}\{\nu_n \geq 1/2\}$ is a sequence of tests as in Definition \ref{def:enh} that furthermore satisfies $\mathbb{P}_{n, d(n), 0}(\nu_n^* = 0) \to 1$ as $n \to \infty$. Hence, requiring existence of tests such that $\mathbb{P}_{n, d(n), 0}(\nu_n = 0) \to 1$ holds instead of \eqref{eqn:asyref} would lead to an equivalent definition. 
\end{remark}

\section{Main question}\label{sec:mainq}

The power enhancement principle tells us how one can improve a sequence of tests $\varphi_n$ provided an enhancement component $\nu_n$ is available. Thus, a natural question now is: when does such an enhancement component actually exist? Similarly, in a situation where there are many possible enhancement components $\nu_n$ available (each improving power against a different sequence $\theta_n$), one can repeatedly apply the power enhancement principle. Then, the question arises: when should one stop enhancing? It is quite tempting to argue that one should keep enhancing until a test is obtained that is not asymptotically enhanceable anymore. But this suggestion is certainly practical only if there exists a test that can not be asymptotically enhanced. This raises the following question:
\begin{center}
\textit{Does there exist a sequence of tests with asymptotic size smaller than one that can not be asymptotically enhanced?}
\end{center}
Note that if the size requirement is dropped in the above question, then the answer is trivially yes, since one can then choose $\varphi_n \equiv 1$, a test that is obviously not asymptotically enhanceable. But this is of no practical use.

The results discussed in Section \ref{sec:Gloc} suggest that the answer to the above question crucially depends on the dimensionality $d(n)$. We first consider the fixed-dimensional case $d(n) \equiv d \in \N$. Here, it turns out that the results from the Gaussian location model discussed in Section \ref{sec:Gloc} are representative in the sense that under weak assumptions there exist sequences of tests that are not asymptotically enhanceable even if the model is not Gaussian. We shall now present a result that supports this claim in the i.i.d.~case under an $\mathbb{L}_2$-differentiability (cf.~Definition 1.103 of \cite{liese}) and a separability condition on the model:
\begin{assumption}\label{as:finited}
For every $n \in \N$ it holds that
\begin{equation}
\Omega_{n, d} = \bigtimes_{i = 1}^n \Omega, \quad \mathcal{A}_{n,d} = \bigotimes_{i = 1}^n \mathcal{A}, \quad \text{and} \quad \mathbb{P}_{n, d, \theta} = \bigotimes_{i = 1}^n \mathbb{P}_{d, \theta} \text{ for every } \theta \in \Theta_d,
\end{equation}
where $\Omega = \Omega_{1, d}$, $\mathcal{A} = \mathcal{A}_{1, d}$ and $\mathbb{P}_{d, \theta} = \mathbb{P}_{1, d, \theta}$. The family $\{\mathbb{P}_{d, \theta}: \theta \in \Theta_d\}$ is $\mathbb{L}_2$-differentiable at $0$ with nonsingular information matrix. Furthermore, for every $\varepsilon > 0$ such that $\Theta_d$ contains a $\theta$ with $\|\theta\|_2 \geq \varepsilon$  there exists a sequence of tests $\psi_{n, d}(\varepsilon): \Omega_{n,d} \to [0, 1]$ such that as $n \to \infty$
\begin{equation}
\mathbb{E}_{n, d, 0}(\psi_{n, d}(\varepsilon)) \to 0 \quad \text{ and } \inf_{\theta \in \Theta_d: \|\theta\|_2 \geq \varepsilon} \mathbb{E}_{n, d, \theta}(\psi_{n, d}(\varepsilon)) \to 1.
\end{equation}
\end{assumption}

Apart from the i.i.d.~condition Assumption \ref{as:finited} is quite weak. In Chapter 10.2 of \cite{van2000asymptotic} a Bernstein-von Mises theorem, attributed to Le Cam, is established under the same set of assumptions. As pointed out in \cite{van2000asymptotic}, a sufficient condition for the existence of tests $\psi_{n,d}(\varepsilon)$ as in Assumption \ref{as:finited} is the existence of a sequence of uniformly consistent estimators $\hat{\theta}_n$ for $\theta$, in which case one can use $\psi_{n,d}(\varepsilon) = \mathbf{1}\{\|\hat{\theta}_n\|_2 \geq \varepsilon/2\}$ (cf.~also the discussion after Theorem \ref{thm:finited2} in Section \ref{sec:finited2}). The proof of the subsequent theorem, which can be found in Section \ref{sec:prooffinited}, combines the observations in Remarks \ref{rem:enh} and \ref{rem:testable} with a minor variation of the argument in the proof of Lemma 10.3 in \cite{van2000asymptotic}. 

\begin{theorem}\label{thm:finited}
Let $d(n) \equiv d$ for some $d\in \N$ and assume that Assumption \ref{as:finited} holds. Then, for every $\alpha \in (0, 1]$ there exists a sequence of tests with asymptotic size $\alpha$ that is not asymptotically enhanceable. 
\end{theorem}
Inspection of the proof shows that the asymptotically unenhanceable test constructed is a combination of a ``truncated score test'' and a suitably chosen sequence of tests $\psi_{n,d}(\varepsilon)$ as in Assumption~\ref{as:finited}. As discussed above, a sufficient condition for sequences of tests $\psi_{n,d}(\varepsilon)$ to exist is the existence of a sequence of uniformly consistent estimators. In case the centered sequence of distributions of these estimators is also uniformly tight over a neighborhood of $0$ in the parameter space (when scaled with the contiguity rate), then Wald-type tests based on this sequence of estimators can be shown to be asymptotically unenhanceable (without imposing Assumption \ref{as:finited}, or an i.i.d.~condition). A formal statement of this result, together with further discussion, is given in Theorem \ref{thm:finited2} in Section~\ref{sec:finited2}.
Theorems \ref{thm:finited} and \ref{thm:finited2} show that one can affirmatively answer the question raised above under weak assumptions in case the non-decreasing sequence $d(n)$ is constant (eventually). Hence we have, in some generality, answered the question raised above for the fixed-dimensional case, and the main question thus is: 
\begin{center}
\textit{Does there exist a sequence of tests with asymptotic size smaller than one that can not be asymptotically enhanced if $d(n)$ diverges with $n$?}
\end{center}

\section{Asymptotic enhanceability in high dimensions}\label{sec:main}

In this section we present our main results concerning the question raised in Section \ref{sec:mainq}. The setting described in Section \ref{sec:fram} is very general and we need to impose some further structural properties on the double array of experiments \eqref{eqn:darray} to answer the question. Our main assumption imposes only a \textit{marginal} local asymptotic normality (LAN) condition on the double array and is as follows:

\begin{assumption}[Marginal LAN]\label{as:lan}
There exists a sequence $s_n > 0$ such that for every \textit{fixed} $d \in \N$ 
\begin{equation}\label{eqn:defHnd}
H_{n, d} := \{h \in \R^d: s_{n}^{-1} h \in \Theta_d\} \uparrow \R^d,
\end{equation}
and such that the sequence of experiments
\begin{equation}\label{eqn:seqexlan}
\mathcal{E}_{n, d} = \left(\Omega_{n, d}, \mathcal{A}_{n, d}, \{\mathbb{P}_{n, d, s_{n}^{-1} h}: h \in H_{n, d} \}\right) \quad \text{ for } n \in \N
\end{equation}
is locally asymptotically normal with positive definite information matrix $\mathsf{I}_d$.\footnote{For completeness, a full statement of the local asymptotic normality condition is recalled from Definition~6.63 of \cite{liese} in the first sentence of Section \ref{sec:proofmainfirststep}.}
\end{assumption}
Note that Assumption \ref{as:lan} only imposes LAN to hold for \textit{fixed} $d$ as $n\to\infty$. Put differently, LAN is only imposed in classical ``fixed-dimensional'' experiments, in which LAN has been verified in many setups as illustrated in the examples in the next section. Frequently $s_{n}$ can be chosen as~$\sqrt{n}$. In principle we could extend our results to situations where $s_n$ is a sequence of invertible matrices that also depends on $d$, but for the sake of simplicity we omit this generalization. 

\subsection{Examples}\label{subsec:ex}
Before we answer the main question of Section \ref{sec:mainq}, we briefly discuss under which additional assumptions our running example satisfies Assumption \ref{as:lan}. Furthermore, we provide several references to other experiments that are LAN for fixed $d$, merely to illustrate the generality of our results.

\begin{continuance}{ex:linreg}\hfill \break
\noindent\textit{Fixed covariates:} Assume that $f$ is absolutely continuous with derivative $f'$ such that $0 < I_f=\int(f'/f)^2dF<\infty$. Suppose further that the double array $\mathsf{X}$ has the following properties: denoting $X_{n,d}=(x_{1,d},...,x_{n,d})'$, for every fixed $d$ and as $n \to \infty$ we have $\frac{1}{n}X_{n,d}'X_{n,d}\to Q_d$ where $Q_d$ has full rank (implying that eventually rank$(X_{n,d})=d$ holds), and $\max_{1\leq i\leq n}\allowbreak(X_{n,d}(X_{n,d}'X_{n,d})^{-1}X_{n,d}')_{ii}\allowbreak\to 0$. It then follows from Theorems 2.3.9 and 2.4.2 in \cite{rieder1994robust} that for every \textit{fixed} $d$ the corresponding sequence of experiments $\mathcal{E}_{n,d}$ in (\ref{eqn:seqexlan}) is LAN  with $s_n=\sqrt{n}$ and $\mathsf{I}_d=I_fQ_d$ being positive definite.

\noindent\textit{Random covariates:} Let the error terms satisfy the same assumptions as in the case of fixed covariates. If, furthermore, for every $d$ the matrix $\mathcal{K}_d=\int xx'dK_d(x)\in \R^{d\times d}$ has full rank $d$, it follows from Theorems 2.3.7 and 2.4.6 in \cite{rieder1994robust} that the corresponding experiment $\mathcal{E}_{n,d}$ in (\ref{eqn:seqexlan}) is LAN for every \textit{fixed} $d$ with $s_n=\sqrt{n}$ and $\mathsf{I}_d=I_f\mathcal{K}_d$ being positive definite.
\end{continuance}

\textit{Further examples:}
Local asymptotic normality for \textit{fixed} $d$ is often satisfied: For example, $\mathbb{L}_2$-differentiable models with i.i.d.~data are covered via Theorem 7.2 in \cite{van2000asymptotic}. Many examples of models being $\mathbb{L}_2$-differentiability and subsequently LAN for fixed $d$, including exponential families, can be found in Chapter 12.2 of \cite{lehmann}, while generalized linear models are covered in \cite{pupashenko2015l2}. Various time series models have been studied in, e.g., \cite{davies}, \cite{swensen1985asymptotic}, \cite{kreiss1987adaptive}, \cite{garel1995local} and \cite{hallin1999local}. For more details and further references on LAN in time series models see also the monographs \cite{dzhaparidze} and \cite{taniguchi2012asymptotic}.

\subsection{Asymptotic enhanceability for ``slowly'' diverging $d(n)$}
We first show that for arrays satisfying Assumption \ref{as:lan} there always exists a range of unbounded sequences $d(n)$ (dimensions of the parameter space) in which \textit{every} test with asymptotic size less than one is asymptotically enhanceable. The proof of this result is based on a generalization of the arguments given in the last paragraph of Section \ref{eq:gauss1} to double arrays of experiments \eqref{eqn:darray} satisfying Assumption \ref{as:lan}. To be precise, we use the following proposition, the proof of which is deferred to Section \ref{sec:proofmain}. 
\begin{proposition}\label{prop:1}
Suppose the double array \eqref{eqn:darray} satisfies Assumption \ref{as:lan}, and for every $d \in \N$ let $v_{1, d}, \hdots, v_{d, d}$ be an orthogonal basis of eigenvectors of $\mathsf{I}_{d}$ such that $v_{i, d}' \mathsf{I}_d v_{i, d} =1 $ for $i = 1, \hdots, d$. Then, there exists a non-decreasing unbounded sequence $p(n)$ in $\N$ and an $M \in \N$, such that for every non-decreasing unbounded sequence of natural numbers $d(n) \leq p(n)$:
\begin{enumerate}
	\item For every $n \geq M$ and $i = 1, \hdots, d(n)$ it holds that 
	\begin{equation}\label{eqn:inclcontained}
	\theta_{i,n} := s_n^{-1} \max\left(\sqrt{\log(d(n))/2}, 1\right)   v_{i,d(n)} \in \Theta_{d(n)},
	\end{equation}
	and every sequence of tests $\varphi_n: \Omega_{n, d(n)} \to [0, 1]$ satisfies
	\begin{equation}
	\mathbb{E}_{n, d(n), 0} (\varphi_n) - d(n)^{-1} \sum_{i = 1}^{d(n)} \mathbb{E}_{n, d(n), \theta_{i,n}} (\varphi_n) \to 0 \quad \text{ as }  n \to \infty.
	\end{equation}
	\item For every sequence $1 \leq i(n) \leq d(n)$ there exists a sequence of tests $\nu_n: \Omega_{n, d(n)} \to [0, 1]$ such that $\mathbb{E}_{n, d(n), 0}(\nu_n) \to 0$ and $\mathbb{E}_{n, d(n), \theta_{i(n),n}}(\nu_n) \to 1$ as $n \to \infty$.
\end{enumerate}
\end{proposition}
As discussed in Section \ref{sec:Gloc}, in the Gaussian location model the first part of Proposition \ref{prop:1} can be verified by an argument given in  \cite{ingster}. Variants of this result in \cite{ingster} have often been used for determining minimax rates for testing problems in Gaussian models, cf., e.g., Proposition 3.12 in \cite{ingster}, the arguments starting at Equation 43 on page 35 in \cite{Lepski2000}, or Section 6.3 of \cite{dumbgen}. For a general discussion of minimax lower bounds and related techniques see, e.g., \cite{tsyba}. 

We now state our first result on asymptotic enhanceability in high-dimensional testing problems. Its proof is given in Section \ref{sec:proofmain2}.
\begin{theorem}\label{thm:main2}
Suppose the double array of experiments \eqref{eqn:darray} satisfies Assumption \ref{as:lan}. Then, there exists a non-decreasing unbounded sequence $p(n)$ in $\N$, such that for any non-decreasing unbounded sequence $d(n)$ in $\N$ satisfying $d(n) \leq p(n)$ every sequence of tests with asymptotic size smaller than one is asymptotically enhanceable. 
\end{theorem}

Recalling the consequences of asymptotic enhanceability of a test already emphasized around Equation \eqref{eqn:PEP} in Section \ref{sec:aden}, we would like to emphasize two implications of Theorem \ref{thm:main2}:

\begin{itemize}
\item Concerning the constructive value of Theorem \ref{thm:main2}: The theorem shows that when the dimension diverges sufficiently slowly (but cf.~also Section \ref{sec:nondecun}) any test of asymptotic size smaller than one can benefit from an application of the power enhancement principle. In particular, \textit{every} such test has \textit{removable} ``blind spots'' of inconsistency. Therefore, if some of them are of major practical relevance, it can be worthwhile to try to remove these via an application of the power enhancement principle. 
\item Theorem \ref{thm:main2} also comes with a distinct warning: when the dimension diverges sufficiently slowly (but, again, cf.~also Section \ref{sec:nondecun}), \emph{every} test with asymptotic size smaller than one is asymptotically enhanceable. In particular, while some ``blind spots'' can be removed by the power enhancement principle, \emph{any} improved test so obtained will still have removable ``blind spots'', as Theorem \ref{thm:main2} applies equally well to the improved test. These ``blind spots'' are test-specific, and are (implicitly or explicitly) determined by the practitioner through the choice of a test. This underscores the importance of carefully selecting the ``right'' test for a specific problem at hand.  
\end{itemize}

Finally, it is also worth noting that while Theorem \ref{thm:main2} guarantees the existence of a power enhancement component, it does not indicate \textit{how} such a sequence of tests can be obtained from an initial sequence of tests $\varphi_n$. Nevertheless, the  proof of Theorem \ref{thm:main2} and Part 2 of Proposition \ref{prop:1} give some insights into how certain enhancement components can be obtained for a given test $\varphi_n$. 

\smallskip

Theorem \ref{thm:main2} shows that every test with asymptotic size less than one is asymptotically enhanceable as long as the dimension of the parameter space diverges sufficiently \textit{slowly}. 
Hence, the results of Theorems \ref{thm:finited} and \ref{thm:finited2} concerning the typical existence of asymptotically unenhanceable tests in the case of $d(n) \equiv d$ do \emph{not} carry over to the case of slowly diverging $d(n)$. This parallels the discussion of the Gaussian location model in Section \ref{sec:Gloc}. Intuition would now suggest that every test must also be asymptotically enhanceable when the dimension of the parameter space increases very \textit{quickly}, as this only makes the testing problem ``more difficult'', thus broadening the scope for increasing the power of a test. As a consequence, one would be led to believe that under the assumptions of Theorem \ref{thm:main2}, the statement in the theorem can be extended to all diverging sequences $d(n)$. However, while correct in the Gaussian location model considered in Section \ref{sec:Gloc}, this is not true in general: asymptotically unenhanceable tests can exist under the assumptions of Theorem \ref{thm:main2} when the dimension of the parameter space increases sufficiently fast. The simple reason is that while Assumption \ref{as:lan} gives enough structure for slowly increasing $d(n)$, it does not impose enough structure for $d(n)$ increasing sufficiently quickly. As one potential consequence, for such $d(n)$, the testing problem can become asymptotically non-testable and hence any test becomes asymptotically unenhanceable for such regimes, cf.~Remark \ref{rem:testable}.\footnote{
Note, however, that Theorem \ref{thm:main2} implies that under Assumption \ref{as:lan} the testing problem is asymptotically testable for all sequences $d(n) \leq p(n)$.} For concreteness consider the following example:

\begin{example}\label{ex:counter}
Let $\Omega_{n, d} = \bigtimes_{i = 1}^n \R^{d}$ and let $\mathcal{A}_{n,d}$ be the Borel sets of $\bigtimes_{i = 1}^n \R^{d}$. Set $\mathbb{P}_{n, d, \theta}$ equal to the $n$-fold product of $N_d(\theta, d^{3} I_d)$, and let $\Theta_d = (-1, 1)^d$. Assumption \ref{as:lan} is obviously satisfied (with $s_n = \sqrt{n}$ and $\mathsf{I}_d = d^{-3} I_d$). We now show that for $d(n) = n$ the testing problem is asymptotically non-testable. It suffices to show that any sequence of tests $\nu_n: \Omega_{n, d(n)} \to [0, 1]$ such that $\lim_{n \to \infty} \mathbb{E}_{n, d(n), 0}(\nu_n) = 0$ must also satisfy
\begin{equation}
\lim_{n \to \infty} \mathbb{E}_{n, d(n), \theta_n}(\nu_n) = 0 \quad \text{ for any sequence } \theta_n \in \Theta_{d(n)}.
\end{equation}
By sufficiency of the vector of sample means, we may assume that $\nu_n$ is a measurable function thereof, which (since $d(n) = n$) is distributed as $N_{n}(\theta, n^2 I_n)$. It hence suffices to verify that the total variation distance between $N_{n}(\theta_n, n^2 I_n)$ and $N_{n}(0, n^2 I_n)$, or equivalently between $N_{n}(n^{-1}\theta_n,  I_n)$ and $N_{n}(0, I_n)$, converges to $0$ as $n \to \infty$. But since each coordinate of $\theta_n$ is bounded in absolute value by $1$, and thus $\|n^{-1} \theta_n \|_2 \leq n^{-1/2} \to 0$, this follows from, e.g., Example 2.3 in \cite{dasgupta}.
\end{example}

A condition that rules out a behavior as in Example \ref{ex:counter} and under which the conclusion of Theorem \ref{thm:main2} carries over to quickly diverging $d(n)$ is discussed in the subsequent section.

\subsection{Asymptotic enhanceability for any non-decreasing unbounded $d(n)$}\label{sec:nondecun}

The testing problem considered in Example \ref{ex:counter} becomes asymptotically non-testable in regimes where $d(n)$ increases too quickly with $n$. Informally speaking, the underlying reason is that increasing $d$ while keeping $n$ fixed leads to a ``loss in information'' in this example. To extend the statement in Theorem \ref{thm:main2} to any non-decreasing unbounded $d(n)$ such a behavior needs to be ruled out, i.e., we need to restrict ourselves to situations where an increase in the amount of data available implies an increase in information. This can be achieved by ensuring that for all triples of natural numbers $d_1 < d_2$ and $n$ the testing problem concerning a zero restriction on the parameter vector in $\left(\Omega_{n, d_1}, \mathcal{A}_{n, d_1}, \{\mathbb{P}_{n, d_1, \theta}: \theta \in \Theta_{d_1}\}\right)$ can be ``embedded'' into the testing problem concerning a zero restriction on the parameter vector in $\left(\Omega_{n, d_2}, \mathcal{A}_{n, d_2}, \{\mathbb{P}_{n, d_2, \theta}: \theta \in \Theta_{d_2}\}\right)$. As a consequence, the testing problem for dimension $d_2$ is then ``more informative'' than the testing problem for dimension $d_1$. To arrive at a mathematically precise condition, we shall consider two testing problems as equivalent if every power function in one experiment is the power function of a test in the other experiment, and vice versa.\footnote{Further discussion and related results concerning the comparison of testing problems based on their informativeness can be found in Chapter 4 of \cite{strasser}. Note that the discussion there is for dominated experiments which we do not require.} The above embedding-idea can then 
formally be stated as follows:

\begin{assumption}\label{as:compfin}
For all pairs of natural numbers $d_1 < d_2$ there exists a function $F = F_{d_1, d_2}$ from $\Theta_{d_1}$ to $\Theta_{d_2}$ satisfying $F(0) = 0$, and such that for every $n \in \N$:
\begin{enumerate}
\item  For every test $\varphi: \Omega_{n, d_2} \to [0, 1]$ there exists a test $\varphi': \Omega_{n, d_1} \to [0, 1]$ such that
\begin{equation}
\mathbb{E}_{n, d_2, F(\theta)}(\varphi) = \mathbb{E}_{n, d_1, \theta}(\varphi') \quad \text{ for every }  \theta \in \Theta_{d_1}.
\end{equation}
\item For every test $\varphi': \Omega_{n, d_1} \to [0, 1]$ there exists a test $\varphi: \Omega_{n, d_2} \to [0, 1]$ such that
\begin{equation}
\mathbb{E}_{n, d_1, \theta}(\varphi') = \mathbb{E}_{n, d_2, F(\theta)}(\varphi) \quad \text{ for every }  \theta \in \Theta_{d_1}.
\end{equation}
\end{enumerate}
\end{assumption}	
The following observation is sometimes useful (e.g., for regression models with fixed covariates or certain time series models) in verifying the preceding assumption.

\begin{remark}\label{rem:nod}
Assumption \ref{as:compfin} is satisfied with $F(\theta) = (\theta', 0)' \in \R^{d_2}$, if for all pairs of natural numbers $d_1 < d_2$, it holds that
\begin{equation}\label{eqn:compsp}
F(\Theta_1) = \Theta_{d_1} \times \{0\}^{d_2-d_1} \subseteq \Theta_{d_2},
\end{equation}
and (i) the sample space does not depend on the dimensionality of the parameter space, i.e., $\Omega_{n, d} = \Omega_n$ and $\mathcal{A}_{n,d} = \mathcal{A}_n$ holds for every $n \in \N$ and every $d \in \N$; and (ii) for all pairs of natural numbers $d_1 < d_d$ the condition $\theta \in \Theta_{d_1}$ implies $\mathbb{P}_{n, d_2, F(\theta)} = \mathbb{P}_{n, d_1, \theta}$. To see this, note that one can then use $\varphi' \equiv \varphi$ in Part 1, and $\varphi \equiv \varphi'$ in Part 2 of Assumption \ref{as:compfin}.
\end{remark}
In our running example, Assumption \ref{as:compfin} holds in the fixed covariates case, and also in the random covariates case under an additional assumption on the family of regressor distributions~$K_d$:
\begin{continuance}{ex:linreg}
Since $\Theta_d=\R^d$ condition \eqref{eqn:compsp} obviously holds. 

\noindent\textit{Fixed covariates:} Since $\Omega_{n,d}$ and $\mathcal{A}_{n,d}$ do not depend on $d$ it follows immediately from the observation in Remark \ref{rem:nod} that Assumption \ref{as:compfin} is satisfied.

\noindent\textit{Random covariates:} In this case further conditions on $K_d$ for $d \in \N$ are necessary. Recall that $K_d$ is a probability measure on the Borel sets of $\R^d$. Given two natural numbers $d_1 < d_2$ associate with $K_{d_2}$ its ``marginal distribution''
\begin{equation}\label{eqn:compK}
K_{d_1, d_2}(A) = K_{d_2}(A \times \R^{d_2-d_1}) \quad \text{ for every Borel set} \quad A \subseteq \R^{d_1}. 
\end{equation}
If for any two natural numbers $d_1 < d_2$ it holds that $K_{d_1} = K_{d_1, d_2}$, then Assumption \ref{as:compfin} is seen to be satisfied by a sufficiency argument; see Section \ref{sec:Exas2} for details.
\end{continuance}
We can now present our final result.
\begin{theorem}\label{thm:main}
Suppose the double array of experiments \eqref{eqn:darray} satisfies Assumptions \ref{as:lan} and \ref{as:compfin}. Then, for every non-decreasing and unbounded sequence $d(n)$ in $\N$ every sequence of tests with asymptotic size smaller than one is asymptotically enhanceable.
\end{theorem}

The proof of this theorem is given in Section \ref{sec:proofthm} where we replace Assumption \ref{as:compfin} with a slightly weaker asymptotic version. While Theorem \ref{thm:main2} establishes the existence of a range of sufficiently slowly non-decreasing unbounded $d(n)$ along which every test is asymptotically enhanceable, Theorem \ref{thm:main} strengthens this property to hold for \textit{any} non-decreasing unbounded $d(n)$. This stronger conclusion comes from adding Assumption \ref{as:compfin}, which now allows one to transfer properties of experiments with slowly increasing $d(n)$ (established through Theorem \ref{thm:main2}) to properties of experiments with quickly increasing $d(n)$.

\section{Conclusion} 
Under weak assumptions, we have shown that there exist tests that are asymptotically unenhanceable in case $d$ is fixed, but that any test of asymptotic size smaller than one is asymptotically enhanceable if $d(n)$ grows to infinity. This latter finding, which constitutes the main insight of this article, reveals that any test possesses removable ``blind spots'' of inconsistency that can be removed by applying the power enhancement principle of \cite{fan2015}. Practitioners are thus forced to prioritize, as any test possesses removable ``blind spots''. It is hence recommended to prioritize deliberately. More specifically, before applying a test one should first analyze its power properties (e.g., numerically) to get an idea about its ``blind spots''. If power is low in regions that are highly ``practically relevant'', the power enhancement principle can provide a way to enhance it. To facilitate such an analysis, theoretically describing removable ``blind spots'' of new or already established tests, in addition to discussing their consistency regions, seems desirable. From a theoretical perspective, interpreting our main contribution as impossibility results, they imply that there are no ``asymptotically optimal tests'' when $d(n)$ grows to infinity, in the sense that the power enhancement principle can always be used to construct an asymptotically better test. In this sense, asymptotic unenhanceability is not generically a reasonable requirement of a test in high-dimensional testing problems. This information should prevent researchers from attempting to do the impossible, i.e., to aim for constructing  asymptotically unenhanceable tests. Accepting that any test has ``blind spots'', it could be an interesting future avenue of research to study whether one can construct tests with ``blind spots'' of a ``minimal'' size, or tests that are asymptotically unenhanceable over specific parameter (sub-)spaces.

\section{Appendix}

Throughout, given a random variable (or vector) $x$ defined on a probability space $(F, \mathcal{F}, \mathbb{Q})$ the image measure induced by $x$ is denoted by $\mathbb{Q} \circ x$. Furthermore, ``$\Rightarrow$'' denotes weak convergence.

\subsection{Additional material for Section~\ref{sec:Gloc}}\label{sec:consreg}

The following lemma shows that the test $\phi_n$ considered in Section~\ref{sec:Gloc} is consistent against $\theta_n$ if and only if $d(n)^{-1/2} n \|\theta_n\|_2^2 \to \infty$. The result is probably well known, but difficult to pinpoint in the literature in this form, and we therefore  provide a direct argument for completeness and for the convenience of the reader.

\begin{lemma}\label{lem:auxGaussloc}
Let~$\alpha \in (0, 1)$, let~$d(n)$ diverge to~$\infty$, and let~$X_1, \dots, X_n$ be i.i.d.~$N_{d(n)}(\theta, I_{d(n)})$. Then, the test~$\phi_n$, which rejects the null hypothesis $H_0: \theta = 0$ if the squared~$\|.\|_2$-norm of~$Z_n = n^{-1/2} \sum_{i = 1}^{n} X_i$ exceeds the~$1-\alpha$ quantile of a~$\chi^2$-distribution with~$d(n)$ degrees of freedom, has~(i) size~$\alpha$ for every~$n \in \N$; and (ii)~is consistent against a sequence~$\theta_n$, where~$\theta_n \in \R^{d(n)}$ for every~$n$, if and only if 
\begin{equation}\label{eqn:cond}
\rho_n := d(n)^{-1/2} n \|\theta_n\|_2^2 \to \infty.
\end{equation}
\end{lemma}

\begin{proof}
Part~(i) is trivial, because under the null ~$\|Z_n\|^2_2$ is~$\chi^2$-distributed with~$d(n)$ degrees of freedom. Consider now Part~(ii): Denote the~$1-\alpha$ quantile of a~$\chi^2$-distribution with~$d(n)$ degrees of freedom by~$\kappa_n$. Observe that~$\phi_n$ rejects if and only if
\begin{equation}\label{eqn:trejequiv}
(\|Z_n\|_2^2 - d(n))/\sqrt{2d(n)} > (\kappa_n - d(n))/\sqrt{2d(n)}.
\end{equation}
It follows immediately from the central limit theorem that under the null~$(\|Z_n\|_2^2 - d(n))/\sqrt{2d(n)} \Rightarrow N(0, 1)$. Consequently, we obtain from Part~(i) that~$(\kappa_n - d(n))/\sqrt{2d(n)}$ must converge to the~$1-\alpha$ quantile of a standard normal distribution,~$\eta$, say. Let~$\theta_n \neq 0$ be a sequence of alternatives. Writing~$ \|Z_n\|_2^2 = \|G_n\|_2^2 + \sqrt{n} 2 G_n' \theta_n + n \|\theta_n\|_2^2$ with~$G_n := Z_n - n^{1/2} \theta_n \sim N_{d(n)}(0, I_{d(n)})$ we have
\begin{equation}\label{eqn:teststat}
(\|Z_n\|_2^2 - d(n))/\sqrt{2d(n)} = (\|G_n\|_2^2 - d(n))/\sqrt{2d(n)} + (\sqrt{n} 2 G_n' \theta_n + n \|\theta_n\|_2^2)/\sqrt{2d(n)}.
\end{equation}
The distribution of the first summand to the right in the previous display does not depend on~$\theta_n$, and converges weakly to~$N(0, 1)$; the second summand to the right  is~$N(\mu_n, \sigma_n^2)$ distributed with
\begin{equation}
\mu_n := 2^{-1/2} \rho_n  \quad \text{ and } \quad \sigma_n^2 := \frac{2}{d^{1/2}(n)} \rho_n.
\end{equation}
To prove sufficiency, suppose that~$\theta_n$ satisfies Equation~\eqref{eqn:cond}. Obviously,~$\phi_n$ rejects if and only if 
\begin{equation}\label{eqn:lastap2norm}
\rho_n^{-1} (\|Z_n\|_2^2 - d(n))/\sqrt{2d(n)} > \rho_n^{-1} (\kappa_n - d(n))/\sqrt{2d(n)}.
\end{equation}
Since~$\rho_n \to \infty$ and because the sequence~$(\kappa_n - d(n))/\sqrt{2d(n)} \to \eta$, as pointed out above, the right hand side converges to~$0$. From Equation \eqref{eqn:teststat} and the observations succeeding it, we conclude that the sequence of random variables to the left in~\eqref{eqn:lastap2norm} converges in probability to~$2^{-1/2}$. This, together with the Portmanteau Theorem, implies that the test under consideration is consistent against~$\theta_n$. Next, we establish necessity: Suppose~$\rho_n$ converges to~$\rho$, say, along a subsequence~$n'$. Then~$N(\mu_n, \sigma_n^2) \Rightarrow \delta_{2^{-1/2} \rho}$ along~$n'$, and by Slutzky's lemma~and~\eqref{eqn:teststat} the sequence of random variables to the left in Equation \eqref{eqn:trejequiv} converges weakly to~$N(2^{-1/2} \rho, 1)$ along~$n'$. From~$(\kappa_n - d(n))/\sqrt{2d(n)} \to \eta$ and the Portmanteau Theorem it then immediately follows that the sequence of tests under consideration is not consistent against such a sequence of alternatives~$\theta_n$.
\end{proof}

\subsection{Proof of Theorem \ref{thm:finited}}\label{sec:prooffinited}

The statement trivially holds for $\alpha = 1$. Let $\alpha \in (0, 1)$. Suppose we could construct a sequence of tests $\varphi^*_n: \Omega_{n, d} \to [0, 1]$ with the property that for some $\varepsilon > 0$ such that $B(\varepsilon)=\cbr[0]{z \in\R^d:\enVert[0]{z}_2<\varepsilon}\subsetneqq \Theta_d$ (recall that $\Theta_d$ is throughout assumed to contain an open neighborhood of the origin) the following holds: $\mathbb{E}_{n, d, 0} (\varphi^*_n) \to \alpha$, and for any sequence $\theta_n \in B(\varepsilon)$ such that $n^{1/2} \| \theta_n \|_2 \to \infty$ it holds that $\mathbb{E}_{n, d, \theta_n} (\varphi^*_n) \to 1$. Given such a sequence of tests, we could define tests $\varphi_n = \min(\varphi_n^* + \psi_{n, d}(\varepsilon), 1)$ (cf.~Assumption \ref{as:finited}), and note that $\varphi_n$ has asymptotic size $\alpha$, and has the property that $\mathbb{E}_{n, d, \theta_n} (\varphi_n) \to 1$ for any sequence $\theta_n \in \Theta_d$ such that $n^{1/2} \|\theta_n\|_2 \to \infty$. But tests with the latter property are certainly not asymptotically enhanceable, because tests $\nu_n: \Omega_{n,d} \to [0, 1]$ can satisfy $\mathbb{E}_{n, d, 0}(\nu_n) \to 0$ and $\mathbb{E}_{n, d, \theta_n}(\nu_n) \to 1$ only if $\theta_n \in \Theta_d$ satisfies $n^{1/2} \|\theta_n\|_2 \to \infty$. To see this use Remark \ref{rem:enh} and recall that convergence of $n^{1/2} \|\theta_n\|_2$ along a subsequence $n'$ together with the maintained i.i.d.~and $\mathbb{L}_2$-differentiability assumption implies contiguity of $\mathbb{P}_{n', d, \theta_{n'}}$ w.r.t.~$\mathbb{P}_{n', d, 0}$ (this can be verified easily using, e.g., results in Section 1.5 of \cite{liese} and Theorem 6.26 in the same reference). It hence remains to construct such a sequence $\varphi_n^*$. To this end, denote by $L: \Omega \to \R^d$ (measurable) an $\mathbb{L}_2$-derivative of $\{ \mathbb{P}_{d, \theta}: \theta \in \Theta_d \}$ at $0$. In the following we denote expectation w.r.t.~$\mathbb{P}_{d, \theta}$ by  $\mathbb{E}_{d, \theta}$. By assumption the information matrix $\mathbb{E}_{d, 0}(LL') = \mathsf{I}_d$ is positive definite. Let $C > 0$ and define $L_C = L \mathbf{1}\{\|L\|_2 \leq C\}$. Since $\mathbb{E}_{d, 0}(L_C L')$ and $M(C)= \mathbb{E}_{d, 0}((L_{C}-\mathbb{E}_{d, 0}(L_{C}))(L_{C}-\mathbb{E}_{d, 0}(L_{C}))')$ converge to $\mathsf{I}_d$ as $C \to \infty$ (by the Dominated Convergence Theorem and $\mathbb{E}_{d, 0}(L) = 0$, for the latter see Proposition 1.110 in \cite{liese}), there exists a $C^*$ such that $\mathbb{E}_{d, 0}(L_{C^*} L')$ and $M := M(C^*)$ are non-singular. Now, by the $\mathbb{L}_2$-differentiability assumption (using again Proposition 1.110 in \cite{liese}), there exists an $\varepsilon > 0$ and a $c > 0$ such that $B(\varepsilon) \subsetneqq \Theta_d$, and such that
\begin{equation}\label{eqn:quadlow}
\|\mathbb{E}_{d, \theta}(L_{C^*}) - \mathbb{E}_{d, 0}(L_{C^*})\|_2 \geq c \|\theta\|_2 \quad \text{ holds for every } \quad \theta \in B(\varepsilon).
\end{equation}
Define on $\bigtimes_{i = 1}^n \Omega$ the functions $Z_n(\theta) := n^{-1/2}\sum_{i = 1}^n(L_{C^*}(\omega_{i,n}) - \mathbb{E}_{d, \theta}(L_{C^*}))$ for $\theta \in \Theta_d$, where $\omega_{i,n}$ denotes the $i$-th coordinate projection on $\bigtimes_{i = 1}^n \Omega$, and set $Z_n(0) = Z_n$. It is easy to verify that $\mathbb{P}_{n, d, \theta_n} \circ Z_n(\theta_n)$ is tight for any sequence $\theta_n \in \Theta_d$, and that by the central limit theorem $\mathbb{P}_{n, d, 0} \circ Z_n \Rightarrow N_d(0, M)$. Finally, let $\varphi_n^*: \Omega_{n, d} \to [0, 1]$ be the indicator function of the set $\{\|Z_n\|_2 \geq Q_{\alpha}\}$, where $Q_{\alpha}$ denotes the $1-\alpha$ quantile of the distribution of the Euclidean norm of a $N_d(0, M)$-distributed random vector. By construction $\mathbb{E}_{n, d, 0}(\varphi_n^*) \to \alpha$. It remains to verify $\mathbb{E}_{n, d, \theta_n} (\varphi^*_n) \to 1$ for any sequence $\theta_n \in B(\varepsilon)$ such that $n^{1/2} \| \theta_n \|_2 \to \infty$. Let $\theta_n$ be such a sequence. By the triangle inequality
\begin{equation}
\|Z_n\|_2 \geq n^{1/2} \| \mathbb{E}_{d, \theta_n}(L_{C^*}) - \mathbb{E}_{d, 0}(L_{C^*})\|_2 - \| Z_n(\theta_n) \|_2.
\end{equation}
Hence, $1-\mathbb{E}_{n, d, \theta_n} (\varphi^*_n)$ is not greater (cf.~\eqref{eqn:quadlow}) than $\mathbb{P}_{n, d, \theta_n}(c n^{1/2} \|\theta_n\|_2 - Q_{\alpha} \leq \| Z_n(\theta_n) \|_2) \to 0$, the convergence following from $\mathbb{P}_{n, d, \theta_n} \circ Z_n(\theta_n)$ being tight, and $c n^{1/2} \|\theta_n\|_2 \to \infty$. 
\qed

\subsection{Theorem \ref{thm:finited2}}\label{sec:finited2}
In this section we present our second result concerning asymptotic enhanceability in the fixed-dimensional case, which was already referred to in Section \ref{sec:mainq}.
\begin{theorem}\label{thm:finited2}
	Let $d(n) \equiv d$ for some $d\in \N$ and let $\|.\|$ be a norm on $\R^d$. Assume that a sequence of estimators $\hat{\theta}_n: \Omega_{n,d} \to \Theta_d$ (measurable) satisfies:
	\begin{enumerate}
		\item Uniform consistency: $\sup_{\theta \in \Theta_d} \mathbb{P}_{n, d, \theta}(\|\hat{\theta}_n - \theta \| > \varepsilon) \to 0$ for every $\varepsilon > 0$.
		\item Contiguity rate: there exists a nondecreasing sequence $s_n > 0$ diverging to $\infty$ such that for every sequence $\theta_n \in \Theta_d$ such that $s_n\|\theta_n \|$ is bounded, the sequence $\mathbb{P}_{n,d,\theta_n}$ is contiguous w.r.t.~$\mathbb{P}_{n,d,0}$.
		\item Local uniform tightness: There exists a $\delta > 0$ such that for every sequence $\theta_n$ in $\Theta_d$ satisfying $\|\theta_n\| \leq \delta$ the sequence of (image) measures $\mathbb{P}_{n,d,\theta_n} \circ [s_n(\hat{\theta}_n - \theta_n)]$ is tight.
	\end{enumerate}
	Then, for every $\alpha \in (0, 1]$ there exists a $C = C(\alpha) \geq 0$ such that the sequence of tests $\varphi_n = \mathbf{1} \{ s_n \|\hat{\theta}\| \geq C\}$ is not asymptotically enhanceable and has asymptotic size not greater than $\alpha$.
\end{theorem}
\begin{proof}
If $\alpha = 1$ set $C = 0$, and note that $\varphi_n := \mathbf{1}\{s_n \|\hat{\theta}_n \| \geq  0\} \equiv 1$, which is obviously not asymptotically enhanceable and has size $1$. Next, consider the case where $\alpha \in (0, 1)$. The existence of a~$C$ ensuring the size requirement follows immediately from the local tightness assumption applied to the sequence $\theta_n = 0$. It remains to show that $\varphi_n := \mathbf{1}\{s_n \|\hat{\theta}_n \| \geq C\}$ is not asymptotically enhanceable. We claim that it suffices to verify that if $s_n \|\theta_n\|$ diverges to $\infty$ for $\theta_n \in \Theta_d$, then $\mathbb{E}_{n,d,\theta_n}(\varphi_n) \to 1$. This claim easily follows from the contiguity rate assumption, together with Remark \ref{rem:enh}.
Now, let $s_n \|\theta_n\|$ diverge to $\infty$. To show that $\mathbb{E}_{n,d,\theta_n}(\varphi_n) \to 1$ it suffices to verify that for every subsequence $n'$ of $n$ there exists a subsequence $n''$ of $n'$ along which $\mathbb{E}_{n,d,\theta_n}(\varphi_n) \to 1$. Let $n'$ be a subsequence of $n$. Then, (i) there exists a subsequence $n''$ of $n'$ such that 
$\|\theta_{n''}\| < \delta$ holds for every $n''$, or (ii) there exists a subsequence $n''$ of $n'$ such that 
$\|\theta_{n''}\| \geq \delta$ holds for every $n''$. Consider first case (i). By the local uniform tightness assumption, the sequence of image measures $\mathbb{P}_{n'',d,\theta_{n''}} \circ [s_{n''}(\hat{\theta}_{n''} - \theta_{n''})]$ is then tight. Let $\varepsilon \in (0, 1)$ and choose $K > 0$ such that $\mathbb{P}_{n'',d,\theta_{n''}} \circ [s_{n''}(\hat{\theta}_{n''} - \theta_{n''})] \left( \bar{B}_{\|.\|}(K) \right) \geq 1-\varepsilon$ holds for every $n''$, where $\bar{B}_{\|.\|}(K) := \{z \in \R^d: \|z\| \leq K\}$. We write
\begin{equation}
\mathbb{E}_{{n''},d,\theta_{n''}}(\varphi_{n''}) = \mathbb{P}_{{n''},d,\theta_{n''}} \circ [s_{n''}(\hat{\theta}_{n''} - \theta_{n''})] \left(\{z \in \R^d: \|z + s_{n''} \theta_{n''}\| \geq C \}\right),
\end{equation}
and note that $\{z \in \R^d: \|z + s_{n''} \theta_{n''}\| \geq C \}$ contains $\bar{B}_{\|.\|}(K)$ for all $n''$ large enough, recalling that $s_n \|\theta_n\| \to \infty$. Hence, the expectation in the previous display is not smaller than $1-\varepsilon$ for $n''$ large enough. Since $\varepsilon$ was arbitrary, it follows that $\mathbb{E}_{n,d,\theta_n}(\varphi_n) \to 1$ along $n''$. Next, we consider the case (ii). In this case, we write
\begin{equation}
\mathbb{E}_{{n''},d,\theta_{n''}}(\varphi_{n''}) = \mathbb{P}_{{n''},d,\theta_{n''}} \left( \| \hat{\theta}_{n''} \| \geq s_{n''}^{-1} C \right) \geq \mathbb{P}_{{n''},d,\theta_{n''}} \left( \| \hat{\theta}_{n''} \| \geq s_{n''}^{-1} C, \| \hat{\theta}_{n''} - \theta_{n''} \| < \delta/2 \right)
\end{equation}
For $n''$ large (since $s_n$ increases to $\infty$ and $\|\theta_{n''}\|\geq \delta$ for every $n''$) the right hand side equals $\mathbb{P}_{{n''},d,\theta_{n''}} ( \| \hat{\theta}_{n''} - \theta_{n''} \| < \delta/2 )$ which converges to $1$ by the uniform consistency assumption. 
\end{proof}

The contiguity rate in Theorem~\ref{thm:finited2} is often given by $s_n = \sqrt{n}$. For an extensive discussion of primitive conditions sufficient for the consistency and tightness assumptions imposed in the previous result we refer the reader to Sections 4 and 5 in Chapter 1 in \cite{ibragimov}, respectively; cf.~also pp.~144-146 in \cite{van2000asymptotic} and Section 5.4 in \cite{pfanzagl2}. We also emphasize that in the i.i.d.~case the local tightness assumption required in Theorem \ref{thm:finited2} is satisfied by the maximum likelihood estimator (MLE) under standard regularity conditions including smoothness and integrability properties of the log-likelihood function over a \emph{neighborhood} of $0$, cf., e.g., the discussion at the end of Section 7 in Chapter 1 in \cite{ibragimov} or the results in Section 7.5 in \cite{pfanzagl} (these regularity conditions, however, are stronger than the $\mathbb{L}_2$-differentiability condition at the point $0$ required by Theorem \ref{thm:finited}; thus Theorem \ref{thm:finited2} is not more general than Theorem \ref{thm:finited} in this respect). In the context of our running example $s_n = \sqrt{n}$ and the OLS estimator satisfies Conditions 1 and 3 in Theorem \ref{thm:finited2} under standard assumptions on the distribution of the errors $F$ and the regressors.

\subsection{Proof of Proposition \ref{prop:1}}\label{sec:proofmain}

The proof is divided into three steps. First we construct a sequence $p(n)$. Then, we verify that the first and second part of Proposition \ref{prop:1}, respectively, is satisfied for this sequence.

\smallskip

\subsubsection{Step 1: Construction of the sequence $p(n)$}\label{sec:proofmainfirststep}

Assumption 2 asserts (cf., Definition~6.63 of \cite{liese}) that for every fixed~$d \in \N$, there exists a sequence of measurable functions (a ``central sequence'') $Z_{n,d}: \Omega_{n,d} \to \R^d$ and a (positive definite and symmetric) information matrix $\mathsf{I}_d$, such that $\mathbb{P}_{n,d,0} \circ Z_n \Rightarrow N_d(0, \mathsf{I}_d)$ (as $n \to \infty$), and such that for every $h \in \R^d$ the (eventually well defined) log-likelihood ratio of $\mathbb{P}_{n,d,s_n^{-1} h}$ w.r.t.~$\mathbb{P}_{n,d,0}$ equals $h'Z_{n,d} -  h' \mathsf{I}_d h/2 + r_{n,d}(h)$ for a measurable sequence $r_{n,d}(h): \Omega_{n,d} \to \bar{\R}$ that converges to $0$ in $\mathbb{P}_{n,d,0}$-probability (as $n \to \infty$). By Theorem 6.76 in \cite{liese}, the following holds for every \textit{fixed} $d \in \N$: there exists a sequence $c(n,d) > 0$ satisfying $c(n,d) \to \infty$ as $n \to \infty$, such that the family of probability measures $\{\mathbb{Q}_{n, d, h}: h \in H_{n, d}\}$ on $(\Omega_{n, d}, \mathcal{A}_{n, d})$ defined via
\begin{equation}\label{eqn:expofam}
\frac{d\mathbb{Q}_{n, d, h}}{d\mathbb{P}_{n, d, 0}} = \exp \left( h' Z^*_{n, d} - K_{n, d}(h) \right),
\end{equation}
where $K_{n, d}(h) = \log(\int_{\Omega_{n,d}} \exp(h'Z^*_{n, d} ) d\mathbb{P}_{n, d, 0})$ and $Z^*_{n, d} = Z_{n, d} \mathbf{1}\{\|Z_{n, d}\|_2 \leq c(n,d) \}$, satisfies
\begin{equation}\label{eqn:Cclose1}
\lim_{n \to \infty}  |K_{n, d}( h) - .5h' \mathsf{I}_{d} h| = 0 \quad \text{ for every } h \in \R^{d},
\end{equation}
and
\begin{equation}\label{eqn:PcloseQ}
\lim_{n \to \infty} d_1(\mathbb{P}_{n, d, s_n^{-1} h}, \mathbb{Q}_{n, d,  h}) = 0 \quad \text{ for every } h \in \R^{d}.
\end{equation}
Here $d_1$ denotes the total variation distance, cf.~\cite{strasser} Definition 2.1. Furthermore (e.g., Theorem 6.72 in \cite{liese}), for every fixed $d \in \N$ and as $n \to \infty$
\begin{equation}\label{eqn:closenormal}
\mathbb{P}_{n, d, s_n^{-1} h} \circ Z_{n, d} \Rightarrow N_{d}(\mathsf{I}_{d} h, \mathsf{I}_{d}) \quad \text{ for every } h \in \R^{d}.
\end{equation}
Next, define the sequence $$a_i = \max\left(\left[.5 \log(i)\right]^{1/2}, 1\right) \quad \text{ for } i \in \N,$$ which (i) is positive, (ii) diverges to $\infty$, and satisfies (iii) $i^{-1} \exp(a_i^2) \to 0$. Now, let $\tilde{H}_d = \left\{0, a_d v_{1, d}, \hdots, a_d v_{d, d}\right\}$ and $H_d = a_d^{-2}\tilde{H}_d \setminus \{0\}$. By $H_{n,d} \uparrow \R^d$ (as $n \to \infty$) and by Equations \eqref{eqn:Cclose1}, \eqref{eqn:PcloseQ}, \eqref{eqn:closenormal} (and the continuous mapping theorem together with $e'\mathsf{I}_de = a_d^{-2}$ for every $e \in H_d$), for every $d \in \N$ there exists an $N(d) \in \N$ such that $n \geq N(d)$ implies (firstly)
\begin{equation}
\tilde{H}_d+\tilde{H}_d \subseteq H_{n, d},
\end{equation}
where, for $A \subseteq \R^d$, the set $A+A$ denotes $\{a + b: a \in A, ~ b \in A\}$, and (secondly)
\begin{align}
&\max_{h \in (\tilde{H}_d+\tilde{H}_d)}|K_{n, d}(h) - .5h' \mathsf{I}_{d} h| + \max_{h \in \tilde{H}_d} d_1(\mathbb{P}_{n, d, s_n^{-1} h}, \mathbb{Q}_{n, d, h}) \\
+ & \max_{(h,e) \in \tilde{H}_d \times  H_d} ~ d_{w}\left(\mathbb{P}_{n, d, s_n^{-1} h} \circ (e'Z_{n, d}), N_{1}(e' \mathsf{I}_{d} h, a_d^{-2})\right) 
\leq  d^{-1}.
\end{align}
Here $d_{w}(.,.)$ denotes a metric on the set of probability measures on the Borel sets of $\R$ that generates the topology of weak convergence, cf.~\cite{dudley} pp.~393 for specific examples. Note also that we can (and do) choose $N(1) < N(2) < \hdots$. Obviously, there exists a non-decreasing unbounded sequence $p(n)$ in $\N$ that satisfies $N(p(n)) \leq n$ for every $n \geq N(1) =: M$. Hence, the two previous displays still hold for $n \geq M$ when $d$ is replaced by $p(n)$. Moreover, the two previous displays also hold for $n \geq M$ when $d$ is replaced by any sequence of non-decreasing natural numbers $d(n) \leq p(n)$. 
The latter implying that for any such sequence $d(n)$ that is also unbounded we have
\begin{equation}\label{eq:contain}
\tilde{H}_{d(n)}+\tilde{H}_{d(n)} \subseteq H_{n, {d(n)}} \quad \text{ for } n \geq M
\end{equation}
and that (as $n \to \infty$)
\begin{equation}\label{eq:Cuniform1}
\max_{h \in (\tilde{H}_{d(n)} + \tilde{H}_{d(n)})} |K_{n, d(n)}(h) - .5h' \mathsf{I}_{d(n)} h| \to 0 
\end{equation}
\begin{equation}\label{eq:tvclose}
\max_{h \in \tilde{H}_{d(n)}} d_1(\mathbb{P}_{n, d(n), s_n^{-1} h}, \mathbb{Q}_{n, d(n), h}) \to 0,
\end{equation}
and
\begin{equation}\label{eq:wclose}
\max_{(h,e) \in \tilde{H}_{d(n)} \times H_{d(n)}}
d_{w}\left(\mathbb{P}_{n, d(n), s_n^{-1} h} \circ (e'Z_{n, d(n)}), N_{1}(e' \mathsf{I}_{d(n)} h, a_{d(n)}^{-2})\right)
\to 0.
\end{equation}
We shall now verify that the sequence $p(n)$ and the natural number $M$ defined above have the required properties. Let $d(n) \leq p(n)$ be an unbounded non-decreasing sequence of natural numbers.

\subsubsection{Step 2: Verification of Part 1}

Equation \eqref{eqn:inclcontained} follows from \eqref{eq:contain} which implies $\tilde{H}_{d(n)} \subseteq H_{n, d(n)}$ for $n \geq M$ (cf.~also Equation~\eqref{eqn:defHnd}). Now, let $\varphi_n: \Omega_{n, d(n)} \to [0, 1]$ be a sequence of tests. For $h \in H_{n, d(n)}$ abbreviate $\mathbb{P}_{n, d(n), s_n^{-1} h} = \mathbb{P}_{n,h}$ and $\mathbb{Q}_{n, d(n), h} = \mathbb{Q}_{n,h}$, and denote expectation w.r.t. $\mathbb{P}_{n,h}$ and $\mathbb{Q}_{n,h}$ by $\mathbb{E}^{P}_{n,h}$ and $\mathbb{E}^{Q}_{n,h}$, respectively. Furthermore, define for $n \geq M$ the probability measures $\mathbb{P}_n = \frac{1}{d(n)} \sum_{h \in \tilde{H}_{d(n)} \setminus \{0\}} \mathbb{P}_{n, h}$, and similarly $\mathbb{Q}_n = \frac{1}{d(n)} \sum_{h \in \tilde{H}_{d(n)} \setminus \{0\}} \mathbb{Q}_{n, h}$. Since for $n \geq M$
\begin{equation}
\big | \mathbb{E}_{n, d(n), 0} (\varphi_n) - d(n)^{-1} \sum_{h \in \tilde{H}_n\setminus \cbr[0]{0}} \mathbb{E}^P_{n, h} (\varphi_n) \big | \leq d_1(\mathbb{P}_{n, 0}, \mathbb{P}_n)
\end{equation}
(cf.~\cite{strasser} Lemma 2.3), it suffices to verify $d_1(\mathbb{P}_{n, 0}, \mathbb{P}_n) \to 0$. From \eqref{eq:tvclose} we see that it suffices to show that $d_{\mathsf{1}}(\mathbb{Q}_{n,0}, \mathbb{Q}_n) \to 0$. Since $\mathbb{Q}_n \ll \mathbb{Q}_{n, 0} = \mathbb{P}_{n,0}$ by \eqref{eqn:expofam},  $d^2_{\mathsf{1}}(\mathbb{Q}_{n,0}, \mathbb{Q}_n)$ equals (e.g., \cite{strasser} Lemma 2.4)
\begin{align}
\left(\frac{1}{2} ~ \mathbb{E}^Q_{n, 0} \left| \frac{d\mathbb{Q}_{n}}{d\mathbb{Q}_{n, 0}} - 1 \right| \right)^2  \leq \mathbb{E}^Q_{n, 0} \left( \frac{d\mathbb{Q}_{n}}{d\mathbb{Q}_{n, 0}} - 1 \right)^2 
= \mathbb{E}^P_{n, 0} \left( \frac{d\mathbb{Q}_{n}}{d\mathbb{P}_{n, 0}}\right)^2 - 1,
\end{align}
the first inequality following from Jensen's inequality. 

It remains to verify that $\limsup\limits_{n \to \infty} \mathbb{E}^P_{n, 0} \left( \frac{d\mathbb{Q}_{n}}{d\mathbb{P}_{n, 0}}\right)^2 \leq 1$: Let $a_{d(n)} = a(n)$, $k_{n, i} = K_{n, d(n)}(a(n) v_{i, d(n)})$, $k_{n, i, j} = K_{n, d(n)}(a(n) v_{i, d(n)} + a(n) v_{j, d(n)})$, and let $z^*_{n, i} = v'_{i, d(n)} Z^*_{n, d(n)}$. Let $n \geq M$. From \eqref{eqn:expofam} we see that
\begin{equation} \frac{d\mathbb{Q}_{n}}{d\mathbb{P}_{n, 0}} = d(n)^{-1} \sum_{i = 1}^{d(n)} \exp(a(n)z^*_{n, i} - k_{n, i})
\end{equation} 
and 
\begin{equation}
\mathbb{E}^P_{n, 0}  \big( \exp\big(a(n)z^*_{n, i} - k_{n, i}\big) \exp\big(a(n)z^*_{n, j} - k_{n, j}\big)  \big)  =  \exp\big(k_{n,i,j} - k_{n,i} - k_{n,j}\big).
\end{equation}
Thus, $\mathbb{E}^P_{n, 0} \left( \frac{d\mathbb{Q}_{n}}{d\mathbb{P}_{n, 0}}\right)^2$ is not greater than the sum of 
\begin{align}
&d(n)^{-1} \exp\big(a^2(n) \big)  \max_{1 \leq i \leq d(n)} 
\exp\big(k_{n,i,i}-2 k_{n,i} - a^2(n)\big) \quad \text{ and }\\
&\max_{1 \leq i < j \leq d(n)} \exp\big(k_{n,i,j} - k_{n,i} - k_{n,j}\big).
\end{align}
But the first sequence converges to $0$, and the second to $1$. This follows from $i^{-1} \exp(a_i^2) \to 0$, and since the sequences $\max\limits_{1\leq i \leq d(n)} |k_{n,i} - .5a^2(n)|$, $\max\limits_{1\leq i \leq d(n)} |k_{n,i, i} - 2a^2(n)|$, and $\max\limits_{1\leq i < j \leq d(n)} |k_{n,i, j} - a^2(n)|$ all converge to $0$ by Equation \eqref{eq:Cuniform1}.

\subsubsection{Step 3: Verification of Part 2}
Given a sequence $1 \leq i(n) \leq d(n)$ define $t_n = a(n)^{-1} v_{i(n), d(n)}' Z_{n, d(n)}$ and let $\nu_n=\mathbf{1}\{ t_n \geq 1/2 \}$. By definition (using the same notation as in Step 2)
\begin{equation}\label{eqn:type1}
\mathbb{E}^P_{n, 0}(\nu_n) = \mathbb{P}_{n, 0} \circ t_n \left([.5, \infty)\right).
\end{equation}
Since $0 \in \tilde{H}_{d(n)}$ and $a(n)^{-1} v_{i(n), d(n)} \in H_{d(n)}$, it follows from \eqref{eq:wclose} that 
\begin{equation}
d_{w} (\mathbb{P}_{n,0} \circ t_n, N_{1}(0, a(n)^{-2})) \to 0.
\end{equation}
But $a(n) \to \infty$ thus implies (via the triangle inequality, together with $d_{w}$-continuity of $(\mu, \sigma^2) \mapsto N_1(\mu, \sigma^2)$ on $\R \times [0, \infty)$, $N_1(\mu, 0)$ being interpreted as $\delta_{\mu}$, i.e., point mass at $\mu$) that $\mathbb{P}_{n,0} \circ t_n \Rightarrow \delta_0$. From the Portmanteau Theorem it hence follows that the sequence in \eqref{eqn:type1} converges to $\delta_0\left([.5, \infty)\right) = 0$. Concerning asymptotic power let $v_n = a(n) v_{i(n), d(n)}$. Note that $v_n \in \tilde{H}_{d(n)}$, $a(n)^{-1} v_{i(n), d(n)} \in H_{d(n)}$ and Equation \eqref{eq:wclose} implies $d_{w}(\mathbb{P}_{n, v_n} \circ t_n, N_{1}(1, a(n)^{-2}))
\to 0$, hence $\mathbb{P}_{n, v_n} \circ t_n \Rightarrow \delta_1$, and thus $\mathbb{E}^P_{n, v_n}(\nu_n) = \mathbb{P}_{n,    v_n} \circ t_n \left([.5, \infty)\right) \to 1$.
\qed

\subsection{Proof of Theorem \ref{thm:main2}}\label{sec:proofmain2}

To prove Theorem \ref{thm:main2}, choose for each $d \in \N$ an arbitrary orthogonal basis as in Proposition \ref{prop:1} to obtain a corresponding sequence $p(n)$, and let $d(n)\leq p(n)$ be non-decreasing and unbounded. Let the sequence of tests~$\varphi_n: \Omega_{n, d(n)} \to [0, 1]$ be of asymptotic size $\alpha <1$, i.e.,  $\limsup_{n \to \infty} \mathbb{E}_{n, d(n), 0}(\varphi_n) = \alpha < 1$. According to Definition \ref{def:enh} we need to show that $\liminf_{n \to \infty} \mathbb{E}_{n, d(n), \theta_n}(\varphi_n) < 1$ for a sequence $\theta_n \in \Theta_{d(n)}$ for which a sequence of tests $\nu_n: \Omega_{n, d(n)} \to [0, 1]$ exists such that
\begin{equation}\label{eqn:findt}
	\lim_{n \to \infty} \mathbb{E}_{n, d(n), 0}(\nu_n) = 0 \quad \text{ and } \quad \lim_{n \to \infty} \mathbb{E}_{n, d(n), \theta_n}(\nu_n) = 1.
\end{equation}
But Part 1 of Proposition \ref{prop:1} implies existence of a sequence $1\leq i(n)\leq d(n)$ such that 
\begin{equation}\label{eqn:blind}
	\limsup_{n \to \infty} \mathbb{E}_{n, d(n), \theta_{i(n), n}} (\varphi_n) \leq \alpha < 1,
\end{equation}
and Part 2 of Proposition \ref{prop:1} verifies existence of a sequence of tests $\nu_n$ as in Equation \eqref{eqn:findt} for~$\theta_n = \theta_{i(n), n}$. \qed

\bigskip

Note that the above proof actually exploits a power enhancement component for a sequence $\theta_n$ against which $\varphi_n$ has asymptotic power not only smaller than one, but in fact at most $\alpha$.

\subsection{Verification of Assumption \ref{as:compfin} for the random covariates case in our running example}\label{sec:Exas2}

We show that Assumption \ref{as:compfin} is satisfied for $F(\theta) = (\theta', 0)' \in \R^{d_2}$. For convenience, denote a generic element of $\Omega_{n,d} = \bigtimes_{	i = 1}^n (\R \times \R^d)$ by $z_{d} = (y, x^{(1)}, \hdots, x^{(d)})$ for $y, x^{(1)}, \hdots, x^{(d)} \in \R^n$. Let $d_1 < d_2$ and $n$ be natural numbers. Consider the experiment 
\begin{equation}\label{eqn:auxex}
	(\Omega_{n, d_2}, \mathcal{A}_{n, d_2}, \{\mathbb{P}_{n, d_2, F(\theta)}: \theta \in \Theta_{d_1}\}),
\end{equation}
define the map $T: \Omega_{n, d_2} \to \Omega_{n, d_1}$ as $T(z_{d_2}) = z_{d_1}$, and note that $T$ is sufficient for \eqref{eqn:auxex} (e.g., Theorem 20.9 in \cite{strasser}). Note further that $\mathbb{P}_{n, d_2, F(\theta)} \circ T = \mathbb{P}_{n, d_1, \theta}$ holds for every $\theta \in \Theta_{d_1}$ under our additional assumption that $K_{d_1} = K_{d_1, d_2}$. That Assumption~\ref{as:compfin} is satisfied now follows from Corollaries 22.4 and 22.6 in \cite{strasser}. 

\subsection{Proof of Theorem \ref{thm:main}} \label{sec:proofthm}
\subsubsection{A weaker version of Assumption \ref{as:compfin}}
Note that Assumption \ref{as:compfin} imposes restrictions to hold for every $n \in \N$. Since asymptotic enhanceability concerns large-sample properties of tests, it is not surprising that a (weaker) asymptotic version of Assumption \ref{as:compfin} suffices for establishing the same conclusion as in Theorem \ref{thm:main}. The asymptotic (and weaker) version of Assumption \ref{as:compfin} we work with subsequently is as follows:
\begin{assumption} \label{as:comp}
For all pairs of natural numbers $d_1 < d_2$ there exists a function $F = F_{d_1, d_2}$ from $\Theta_{d_1}$ to $\Theta_{d_2}$ satisfying $F(0) = 0$, and such that for any two non-decreasing unbounded sequences $r(n)$ and $d(n)$ in $\N$ such that $r(n) < d(n)$ the following holds, abbreviating $F_{r(n), d(n)}$ by $F_n$:
\begin{enumerate}
	\item For every sequence of tests $\varphi_n: \Omega_{n, d(n)} \to [0, 1]$, there exists a sequence of tests $\varphi_n': \Omega_{n, r(n)} \to [0, 1]$ such that
	\begin{equation}\label{eqn:compEQ1}
	\sup_{\theta \in \Theta_{r(n)}}  \left| \mathbb{E}_{n, d(n), F_n(\theta)}(\varphi_n)  - \mathbb{E}_{n, r(n), \theta}(\varphi_n') \right| \to 0 \text{ as } n \to \infty.
	\end{equation}
	\item For every sequence of tests $\varphi_n': \Omega_{n, r(n)} \to [0, 1]$, there exists a sequence of tests $\varphi_n: \Omega_{n, d(n)} \to [0, 1]$ such that
	\begin{equation}\label{eqn:compEQ2}
	\sup_{\theta \in \Theta_{r(n)}} \left| \mathbb{E}_{n, r(n), \theta}(\varphi_n') - \mathbb{E}_{n, d(n), F_n(\theta)}(\varphi_n) \right| \to 0 \text{ as } n \to \infty.
	\end{equation}
\end{enumerate}
\end{assumption}

\subsubsection{Proof of Theorem \ref{thm:main}}

We shall now prove the conclusion of Theorem \ref{thm:main} under slightly weaker conditions by replacing Assumption \ref{as:compfin} by Assumption \ref{as:comp}. Theorem \ref{thm:main} then follows immediately as a Corollary.
\begin{theorem}\label{thm:main_general}
Suppose the double array of experiments \eqref{eqn:darray} satisfies Assumptions \ref{as:lan} and \ref{as:comp}. Then, for every non-decreasing and unbounded sequence $d(n)$ in $\N$ every sequence of tests with asymptotic size smaller than one is asymptotically enhanceable.
\end{theorem}
\begin{proof}
Let $d(n)$ be a non-decreasing and unbounded sequence in $\N$, and let $\varphi_n: \Omega_{n, d(n)} \to [0, 1]$ be of asymptotic size $\alpha < 1$.
We apply Theorem \ref{thm:main2} to obtain a sequence $p(n)$ as in that theorem. Let $r(n) \equiv \min(p(n),d(n)-1)$, a non-decreasing unbounded sequence that eventually satisfies $r(n) \in \N$ and $r(n) < d(n)$. By Part 1 of Assumption \ref{as:comp} there exists a sequence of tests $\varphi_n': \Omega_{n, r(n)} \to [0, 1]$ such that \eqref{eqn:compEQ1} holds. In particular $\varphi_n'$ also has asymptotic size $\alpha$, recalling that $F_n(0) = 0$ holds by assumption.  Therefore, by Theorem \ref{thm:main2} (applied  with ``$d(n) \equiv r(n)$''), $\varphi_n'$ is asymptotically enhanceable, i.e., there exist tests $\nu_n': \Omega_{n, r(n)} \to [0, 1]$ and a sequence $\theta_n \in \Theta_{r(n)}$ such that $\mathbb{E}_{n, r(n), 0}(\nu_n') \to 0$ and
\begin{equation}
1 = \lim_{n \to \infty} \mathbb{E}_{n, r(n), \theta_n}(\nu_n') > \liminf_{n \to \infty} \mathbb{E}_{n, r(n), \theta_n}(\varphi_n') = \liminf_{n \to \infty} \mathbb{E}_{n, d(n), F_n(\theta_n)}(\varphi_n),
\end{equation}
the second equality following from \eqref{eqn:compEQ1}. By Part 2 of Assumption \ref{as:comp}, and using again $F_n(0) = 0$, tests $\nu_n: \Omega_{n, d(n)} \to [0, 1]$ exist such that $\mathbb{E}_{n, d(n), 0}(\nu_n) \to 0$ and $\mathbb{E}_{n, d(n), F_n(\theta_n)}(\nu_n) \to 1$. Hence $\varphi_n$ is asymptotically enhanceable. 
\end{proof}

\bibliographystyle{ims}
\bibliography{refs}

\begin{thebibliography}{45}
\expandafter\ifx\csname natexlab\endcsname\relax\def\natexlab#1{#1}\fi
\expandafter\ifx\csname url\endcsname\relax
  \def\url#1{\texttt{#1}}\fi
\expandafter\ifx\csname urlprefix\endcsname\relax\def\urlprefix{URL }\fi
\providecommand{\eprint}[2][]{\url{#2}}

\bibitem[{Abadie and Kasy(2018)}]{kasy}
\textsc{Abadie, A.} and \textsc{Kasy, M.} (2018).
\newblock Choosing among regularized estimators in empirical economics: The
  risk of machine learning.
\newblock \textit{Review of Economics and Statistics} forthcoming.

\bibitem[{Bai et~al.(2009)Bai, Jiang, Yao and Zheng}]{baijiang2009}
\textsc{Bai, Z.}, \textsc{Jiang, D.}, \textsc{Yao, J.-F.} and \textsc{Zheng,
  S.} (2009).
\newblock Corrections to {LRT} on large-dimensional covariance matrix by {RMT}.
\newblock \textit{Annals of Statistics}, \textbf{37} 3822--3840.

\bibitem[{Bai and Saranadasa(1996)}]{bai1996}
\textsc{Bai, Z.} and \textsc{Saranadasa, H.} (1996).
\newblock Effect of high dimension: by an example of a two sample problem.
\newblock \textit{Statistica Sinica} 311--329.

\bibitem[{Cai et~al.(2013)Cai, Fan and Jiang}]{cai2013distributions}
\textsc{Cai, T.}, \textsc{Fan, J.} and \textsc{Jiang, T.} (2013).
\newblock Distributions of angles in random packing on spheres.
\newblock \textit{Journal of Machine Learning Research}, \textbf{14}
  1837--1864.

\bibitem[{Cai et~al.(2014)Cai, Liu and Xia}]{tony2014two}
\textsc{Cai, T.}, \textsc{Liu, W.} and \textsc{Xia, Y.} (2014).
\newblock Two-sample test of high dimensional means under dependence.
\newblock \textit{Journal of the Royal Statistical Society: Series B
  (Statistical Methodology)}, \textbf{76} 349--372.

\bibitem[{Chakraborty and Chaudhuri(2017)}]{chakraborty2017}
\textsc{Chakraborty, A.} and \textsc{Chaudhuri, P.} (2017).
\newblock Tests for high-dimensional data based on means, spatial signs and
  spatial ranks.
\newblock \textit{Annals of Statistics}, \textbf{45} 771--799.

\bibitem[{Cutting et~al.(2017)Cutting, Paindaveine and
  Verdebout}]{cutting2017testing}
\textsc{Cutting, C.}, \textsc{Paindaveine, D.} and \textsc{Verdebout, T.}
  (2017).
\newblock Testing uniformity on high-dimensional spheres against monotone
  rotationally symmetric alternatives.
\newblock \textit{Annals of Statistics}, \textbf{45} 1024--1058.

\bibitem[{DasGupta(2008)}]{dasgupta}
\textsc{DasGupta, A.} (2008).
\newblock \textit{Asymptotic Theory of Statistics and Probability}.
\newblock Springer.

\bibitem[{Davies(1973)}]{davies}
\textsc{Davies, R.~B.} (1973).
\newblock Asymptotic inference in stationary {G}aussian time-series.
\newblock \textit{Advances in Applied Probability}, \textbf{5} 469–497.

\bibitem[{Dempster(1958)}]{dempster}
\textsc{Dempster, A.~P.} (1958).
\newblock A high dimensional two sample significance test.
\newblock \textit{Annals of Mathematical Statistics}, \textbf{29} 995--1010.

\bibitem[{Dudley(2002)}]{dudley}
\textsc{Dudley, R.~M.} (2002).
\newblock \textit{Real Analysis and Probability}.
\newblock Cambridge University Press.

\bibitem[{D\"umbgen and Spokoiny(2001)}]{dumbgen}
\textsc{D\"umbgen, L.} and \textsc{Spokoiny, V.~G.} (2001).
\newblock Multiscale testing of qualitative hypotheses.
\newblock \textit{Annals of Statistics} 124--152.

\bibitem[{Dzhaparidze(1986)}]{dzhaparidze}
\textsc{Dzhaparidze, K.} (1986).
\newblock \textit{Parameter Estimation and Hypothesis Testing in Spectral
  Analysis of Stationary Time Series}.
\newblock Springer.

\bibitem[{Fan et~al.(2015)Fan, Liao and Yao}]{fan2015}
\textsc{Fan, J.}, \textsc{Liao, Y.} and \textsc{Yao, J.} (2015).
\newblock Power enhancement in high-dimensional cross-sectional tests.
\newblock \textit{Econometrica}, \textbf{83} 1497--1541.

\bibitem[{Garel and Hallin(1995)}]{garel1995local}
\textsc{Garel, B.} and \textsc{Hallin, M.} (1995).
\newblock Local asymptotic normality of multivariate {A}{R}{M}{A} processes
  with a linear trend.
\newblock \textit{Annals of the Institute of Statistical Mathematics},
  \textbf{47} 551--579.

\bibitem[{Hallin et~al.(1999)Hallin, Taniguchi, Serroukh and
  Choy}]{hallin1999local}
\textsc{Hallin, M.}, \textsc{Taniguchi, M.}, \textsc{Serroukh, A.} and
  \textsc{Choy, K.} (1999).
\newblock Local asymptotic normality for regression models with long-memory
  disturbance.
\newblock \textit{Annals of Statistics}, \textbf{27} 2054--2080.

\bibitem[{Ibragimov and Has'minskii(1981)}]{ibragimov}
\textsc{Ibragimov, I.~A.} and \textsc{Has'minskii, R.~Z.} (1981).
\newblock \textit{Statistical Estimation}.
\newblock Springer.

\bibitem[{Ingster and Suslina(2003)}]{ingster}
\textsc{Ingster, Y.} and \textsc{Suslina, I.~A.} (2003).
\newblock \textit{Nonparametric goodness-of-fit testing under Gaussian models}.
\newblock Springer.

\bibitem[{Janssen(2000)}]{janssen2000}
\textsc{Janssen, A.} (2000).
\newblock Global power functions of goodness of fit tests.
\newblock \textit{Annals of Statistics}, \textbf{28} 239--253.

\bibitem[{Kreiss(1987)}]{kreiss1987adaptive}
\textsc{Kreiss, J.-P.} (1987).
\newblock On adaptive estimation in stationary {A}{R}{M}{A} processes.
\newblock \textit{Annals of Statistics} 112--133.

\bibitem[{Ledoit and Wolf(2002)}]{ledoit2002some}
\textsc{Ledoit, O.} and \textsc{Wolf, M.} (2002).
\newblock Some hypothesis tests for the covariance matrix when the dimension is
  large compared to the sample size.
\newblock \textit{Annals of Statistics}, \textbf{30} 1081--1102.

\bibitem[{Lehmann and Romano(2006)}]{lehmann}
\textsc{Lehmann, E.~L.} and \textsc{Romano, J.~P.} (2006).
\newblock \textit{Testing Statistical Hypotheses}.
\newblock Springer.

\bibitem[{Lepski and Tsybakov(2000)}]{Lepski2000}
\textsc{Lepski, O.} and \textsc{Tsybakov, A.} (2000).
\newblock Asymptotically exact nonparametric hypothesis testing in sup-norm and
  at a fixed point.
\newblock \textit{Probability Theory and Related Fields}, \textbf{117} 17--48.

\bibitem[{Ley et~al.(2015)Ley, Paindaveine and Verdebout}]{ley2015}
\textsc{Ley, C.}, \textsc{Paindaveine, D.} and \textsc{Verdebout, T.} (2015).
\newblock High-dimensional tests for spherical location and spiked covariance.
\newblock \textit{Journal of Multivariate Analysis}, \textbf{139} 79 -- 91.

\bibitem[{Liese and Miescke(2008)}]{liese}
\textsc{Liese, F.} and \textsc{Miescke, K.~J.} (2008).
\newblock \textit{Statistical Decision Theory}.
\newblock Springer.

\bibitem[{Lockhart(2016)}]{lockhart2016}
\textsc{Lockhart, R.~A.} (2016).
\newblock Inefficient best invariant tests.
\newblock \textit{arXiv preprint arXiv:1608.05994}.

\bibitem[{Onatski et~al.(2013)Onatski, Moreira and Hallin}]{onatski2013}
\textsc{Onatski, A.}, \textsc{Moreira, M.~J.} and \textsc{Hallin, M.} (2013).
\newblock Asymptotic power of sphericity tests for high-dimensional data.
\newblock \textit{Annals of Statistics}, \textbf{41} 1204--1231.

\bibitem[{Onatski et~al.(2014)Onatski, Moreira and Hallin}]{onatski2014}
\textsc{Onatski, A.}, \textsc{Moreira, M.~J.} and \textsc{Hallin, M.} (2014).
\newblock Signal detection in high dimension: The multispiked case.
\newblock \textit{Annals of Statistics}, \textbf{42} 225--254.

\bibitem[{Pfanzagl(1994)}]{pfanzagl}
\textsc{Pfanzagl, J.} (1994).
\newblock \textit{Parametric Statistical Theory}.
\newblock de Gruyter.

\bibitem[{Pfanzagl(2017)}]{pfanzagl2}
\textsc{Pfanzagl, J.} (2017).
\newblock \textit{Mathematical Statistics: Essays on History and Methodology}.
\newblock Springer.

\bibitem[{Pinelis(2010)}]{pinelis2010asymptotic}
\textsc{Pinelis, I.} (2010).
\newblock Asymptotic efficiency of p-mean tests for means in high dimensions.
\newblock \textit{arXiv preprint arXiv:1006.0505}.

\bibitem[{Pinelis(2014)}]{pinelisI}
\textsc{Pinelis, I.} (2014).
\newblock Schur 2-concavity properties of {G}aussian measures, with
  applications to hypotheses testing.
\newblock \textit{Journal of Multivariate Analysis}, \textbf{124} 384 -- 397.

\bibitem[{Pupashenko et~al.(2015)Pupashenko, Ruckdeschel and
  Kohl}]{pupashenko2015l2}
\textsc{Pupashenko, D.}, \textsc{Ruckdeschel, P.} and \textsc{Kohl, M.} (2015).
\newblock L2 differentiability of generalized linear models.
\newblock \textit{Statistics \& Probability Letters}, \textbf{97} 155--164.

\bibitem[{Rieder(1994)}]{rieder1994robust}
\textsc{Rieder, H.} (1994).
\newblock \textit{Robust Asymptotic Statistics}.
\newblock Springer.

\bibitem[{Srivastava(2005)}]{srivastava2005}
\textsc{Srivastava, M.~S.} (2005).
\newblock Some tests concerning the covariance matrix in high dimensional data.
\newblock \textit{Journal of the Japan Statistical Society}, \textbf{35}
  251--272.

\bibitem[{Srivastava and Du(2008)}]{srivastava2008}
\textsc{Srivastava, M.~S.} and \textsc{Du, M.} (2008).
\newblock A test for the mean vector with fewer observations than the
  dimension.
\newblock \textit{Journal of Multivariate Analysis}, \textbf{99} 386 -- 402.

\bibitem[{Srivastava et~al.(2013)Srivastava, Katayama and
  Kano}]{srivastava2013}
\textsc{Srivastava, M.~S.}, \textsc{Katayama, S.} and \textsc{Kano, Y.} (2013).
\newblock A two sample test in high dimensional data.
\newblock \textit{Journal of Multivariate Analysis}, \textbf{114} 349 -- 358.

\bibitem[{Steinberger(2016)}]{steinberger2016}
\textsc{Steinberger, L.} (2016).
\newblock The relative effects of dimensionality and multiplicity of hypotheses
  on the {F}-test in linear regression.
\newblock \textit{Electronic Journal of Statistics}, \textbf{10} 2584--2640.

\bibitem[{Strasser(1985)}]{strasser}
\textsc{Strasser, H.} (1985).
\newblock \textit{Mathematical Theory of Statistics}.
\newblock Walter de Gruyter.

\bibitem[{Swensen(1985)}]{swensen1985asymptotic}
\textsc{Swensen, A.~R.} (1985).
\newblock The asymptotic distribution of the likelihood ratio for
  autoregressive time series with a regression trend.
\newblock \textit{Journal of Multivariate Analysis}, \textbf{16} 54--70.

\bibitem[{Taniguchi and Kakizawa(2000)}]{taniguchi2012asymptotic}
\textsc{Taniguchi, M.} and \textsc{Kakizawa, Y.} (2000).
\newblock \textit{Asymptotic Theory of Statistical Inference for Time Series}.
\newblock Springer.

\bibitem[{Tsybakov(2009)}]{tsyba}
\textsc{Tsybakov, A.~B.} (2009).
\newblock \textit{Introduction to Nonparametric Estimation}.
\newblock Springer.

\bibitem[{van~der Vaart(2000)}]{van2000asymptotic}
\textsc{van~der Vaart, A.~W.} (2000).
\newblock \textit{Asymptotic Statistics}.
\newblock Cambridge University Press.

\bibitem[{Wang and Cui(2013)}]{wang2013}
\textsc{Wang, S.} and \textsc{Cui, H.} (2013).
\newblock Generalized {F} test for high dimensional linear regression
  coefficients.
\newblock \textit{Journal of Multivariate Analysis}, \textbf{117} 134--149.

\bibitem[{Zhong and Chen(2011)}]{zhong2011}
\textsc{Zhong, P.-S.} and \textsc{Chen, S.~X.} (2011).
\newblock Tests for high-dimensional regression coefficients with factorial
  designs.
\newblock \textit{Journal of the American Statistical Association},
  \textbf{106} 260--274.

\end{thebibliography}

\end{document}